\documentclass[a4paper,11pt]{article}
\usepackage{amsfonts}
\usepackage{amsmath,amsthm}
\usepackage{amssymb,amscd}
\usepackage{enumerate}
\usepackage{color}
\usepackage{graphicx}
\usepackage[inline]{enumitem}
\usepackage{fullpage}
\usepackage{hyperref}
\usepackage{url}

\makeatletter

\newcommand{\Rmnum}[1]{\expandafter\@slowromancap\romannumeral #1@}
\makeatother
\def\Xint#1{\mathchoice
{\XXint\displaystyle\textstyle{#1}}%
{\XXint\textstyle\scriptstyle{#1}}%
{\XXint\scriptstyle\scriptscriptstyle{#1}}%
{\XXint\scriptscriptstyle\scriptscriptstyle{#1}}%
\!\int}
\def\XXint#1#2#3{{\setbox0=\hbox{$#1{#2#3}{\int}$}
\vcenter{\hbox{$#2#3$}}\kern-.5\wd0}}

\def\dashint{\Xint-}

\newcommand{\compC}{\mathbb{C}}
\newcommand{\realR}{\mathbb{R}}
\newcommand{\intZ}{\mathbb{Z}}
\newcommand{\bigO}{\mathcal{O}}
\newcommand{\Prob}{\mathbb{P}}
\newcommand{\K}{\mathcal{K}}
\newcommand{\dettwo}{\det\nolimits_2}
\newcommand{\Peche}{P\'{e}che\'{e}}
\newcommand{\Weylchamber}[1][N]{\mathbb{W}_{#1}}

\newcommand{\prodprime}{\sideset{}{'}\prod}
\newcommand{\qbinom}[2]{\genfrac{[}{]}{0pt}{}{#1}{#2}_q}

\DeclareMathOperator{\tr}{tr}
\DeclareMathOperator{\loc}{loc}
\DeclareMathOperator{\res}{res}
\DeclareMathOperator{\step}{step}
\DeclareMathOperator{\qTASEP}{\ensuremath{q}-TASEP}
\DeclareMathOperator{\GUE}{GUE}
\DeclareMathOperator{\BBP}{BBP}

\def\Xint#1{\mathchoice
{\XXint\displaystyle\textstyle{#1}}%
{\XXint\textstyle\scriptstyle{#1}}%
{\XXint\scriptstyle\scriptscriptstyle{#1}}%
{\XXint\scriptscriptstyle\scriptscriptstyle{#1}}%
\!\int}
\def\XXint#1#2#3{{\setbox0=\hbox{$#1{#2#3}{\int}$}
\vcenter{\hbox{$#2#3$}}\kern-.5\wd0}}

\def\dashint{\Xint-}

\newtheorem{theorem}{Theorem}[section]

\newtheorem{corollary}[theorem]{Corollary}

\newtheorem{proposition}[theorem]{Proposition}

\theoremstyle{remark}
\newtheorem{remark}{Remark}[section]

\theoremstyle{definition}

\begin{document}
\title{Distributions of a particle's position and their asymptotics in the $q$-deformed totally asymmetric zero range process with site dependent jumping rates}

\author{Eunghyun Lee\thanks{Department of mathematics, Nazarbayev University, Astana, 010000 Kazakhstan, \url{eunghyun.lee@nu.edu.kz}, Supported by the Social Policy Grant of Nazarbayev University.}  \and Dong Wang\thanks{Department of Mathematics, National University of Singapore, 119076 Singapore, \url{matwd@nus.edu.sg}. Support partially by the Singapore AcRF Tier 1 grant R-146-000-217-112.}}
\maketitle

\begin{abstract}
\noindent  In this paper we study the probability distribution of the position of a tagged particle in the $q$-deformed Totally Asymmetric Zero Range Process ($q$-TAZRP) with site dependent jumping rates. For a finite particle system, it is derived from the transition probability previously obtained by Wang and Waugh. We also provide the probability distribution formula for a tagged particle in the $q$-TAZRP with the so-called step initial condition in which infinitely many particles occupy one single site and all other sites are unoccupied. For the $q$-TAZRP with step initial condition, we provide a Fredholm determinant representation for the probability distribution function of the position of a tagged particle, and moreover we obtain the limiting distribution function as the time goes to infinity. Our asymptotic result for $q$-TAZRP with step initial condition is comparable to the limiting distribution function obtained by Tracy and Widom for the $k$-th leftmost particle in the asymmetric simple exclusion process with step initial condition (Theorem 2 in Commun. Math. Phys. 290, 129--154 (2009)).
\end{abstract}

\section{Introduction}

The Zero Range Process (ZRP) is the system of identical particles on a countable set $S$ whose dynamics is described by the following rules (if the sites are homogeneous): if a site $x\in S$ is occupied by $k$ particles, then one of $k$ particles leaves $x$ after an exponential waiting time with rate $g(k)$. The particle leaving $x$ chooses a target site $y$ with probability $p(x,y)$ and then it immediately moves to $y$.  As a special case of the ZRP, the system of $N$ particles on $S=\mathbb{Z}$ with $p(x,x+1) = 1$ and $g(k) = (1-q^{k})(1-q)$ with $q\in (0,1)$, is an \textit{integrable} model originated from the $q$-boson model introduced by Sasamoto and Wadati \cite{Sasamoto-Wadati98}. In this paper, we will call this special case of the ZRP the (spatial-homogeneous) $q$-deformed Totally Asymmetric Zero Range Process ($q$-TAZRP). The integrability of the $q$-TAZRP was studied by Povolotsky \cite{Povolotsky04} and its transition probability was obtained by Korhonen and Lee \cite{Korhonen-Lee14}. The $q$-TAZRP is also a special case of the zero-range chipping model introduced by Povolotsky \cite{Povolotsky13} in which multiple particles are allowed to jump at the same time and a jumping rate depends on the number of particles jumping as well as the  number of particles at the departing site. The $q$-TAZRP can be generalized to let the jumping rates depend on not only the number of particles at the departure site, but also the site $x$ itself, so that the jumping rate at $x$ occupied by $k$ particles is given by
\begin{equation} \label{eq-rate}
  g (k; x) = a_x(1-q^{k}).
\end{equation}
The finite system with the rate function \eqref{eq-rate} is well defined if $a_x > 0$, and we assume that all $a_x$ are in a compact subset of $(0,\infty)$. The transition probability of this model with finitely many particles was obtained by Wang and Waugh \cite{Wang-Waugh16} as an extension of Korhonen and Lee's work \cite{Korhonen-Lee14}. We note that the transition probability can also be derived as a degeneration of the higher spin stochastic six vertex model studied by Borodin and Petrov \cite{Borodin-Petrov16}. Although the $q$-TAZRP with spatial-inhomogeneous jumping rate (that we refer to simply as $q$-TAZRP throughout the paper) is not much studied in literature, its dual model, the $q$-deformed Totally Asymmetric Simple Exclusion Process ($q$-TASEP) with particle-dependent jumping rate, has been extensively studied in \cite{Barraquand15},\cite{Borodin-Corwin-Sasamoto14},  and \cite{Ferrari-Veto15}.

The aim of this paper is two-fold: First for an arbitrary initial condition in the $N$-particle $q$-TAZRP, we find the explicit distribution function for the position of a tagged particle (as the $n$-th rightmost or leftmost particle) at time $t > 0$ if the jumping rates are given by \eqref{eq-rate}. In \cite{Wang-Waugh16} the transition probability for  an $N$-particle system was obtained from which we derive the distribution of a tagged particle's position as a  marginal probability. The result is given in Theorem \ref{main-theorem-1}. We also find the distribution function for a tagged particle's position under the so-called step initial condition which is such that initially infinitely many particles occupy a single site and all other sites are empty, as shown in Theorem \ref{theorem-step}. Under the step initial condition, the distribution of a tagged particle's position is represented by a contour integral of the Fredholm determinant of a certain trace class operator. We also study the asymptotic behaviour of this Fredholm determinant representation when $t$ goes to infinity.  This result is given in Theorem \ref{thm:asymp}. We note that although the distribution of the $m$-th rightmost particle at time $t$ in the limit that both $m$ and $t$ go to infinity can be obtained by the known results for the $q$-TASEP, as discussed in Appendix \ref{sec:TW_limit}, the limiting distribution when $m$ is fixed is new. Since our limiting distribution agrees with that for the Asymmetric Simple Exclusion Process (ASEP) (\cite[Theorem 2]{Tracy-Widom09}), these results may describe a new universal behaviour of interacting particle systems.

\subsection{Definition of the Model} \label{subsec:defn_of_model}

First, we consider a system that consists of $N$ identical particles on $\mathbb{Z}$.  We label particles $1,\dotsc,N$ from the rightmost particle. We assume that if more than one particle occupy the same position, they are ordered vertically such that the label of a particle at a higher position is smaller than the label of a particle at a lower position. We interpret the particle labelled $m$ as the $m$-th rightmost particle. We denote by $x_m(t)$ the position of the $m$-th rightmost particle at time $t$, and $t$ may be omitted if it is not necessary to specify the time. By abuse of notation, we also use $x_m$ to denote the $m$-th particle itself. The state of an $N$-particle system is represented by
\begin{equation} \label{state-1}
  X=(x_1, \dotsc, x_N) \in \Weylchamber := \{(i_1,\dotsc, i_N) \in \mathbb{Z}^N \mid i_1 \geq i_2 \geq \dotsb \geq i_N \}
\end{equation}
and a state at time $t$ is denoted by $X(t) = (x_1(t),  \dotsc, x_N(t))$. Alternatively, a state can be specified by the number of particles at each site. We denote the number of particles at site $k$ (at time $t$) by $n_k$ (or $n_k(t)$).  Hence, a state $X(t)$ is equivalently expressed as $(n_k(t))_{k \in \intZ}$, which satisfies $\sum^{\infty}_{k = -\infty} n_k(t) = N$.

The rules of the model are a specialization of the rules for the general ZRP explained in the beginning of this paper. But since we label the otherwise identical particles, the rules can be described more concretely. We fix $q \in (0,1)$. If a site $x$ is occupied by $k$ particles, then the particle at the top of the particles (the particle itself if a site is occupied by only one particle) jumps to the bottom of the particles at $x+1$ (if there are other particles at $x + 1$) after an exponential time with rate $a_x(1-q^{k})$ where $a_x>0$ for all $x \in \mathbb{Z}$.  In this procedure only the particle at the top may jump after an exponential time and other particles do not respond to the exponential clock. If the local state at $x$ is changed, that is, if the number of particles at $x$ is changed from $k$ to $k'$, then the exponential clock is reset to have the new rate $a_x(1-q^{k'})$ and the particle at the top jumps after a new exponential time.  Also, we assume that all waiting times are independent. We denote by $\Prob_Y$ the probability measure with an initial state $X(0) = Y = (y_1,\dotsc, y_N)$ of the process.

We also consider an infinite system as a limiting case of $N$-particle system. An infinite particle $q$-TAZRP is well defined by \cite[Section 2]{Holley70} if we let the transition rules there be specified as $P(x, y)$ in \cite{Holley70} is equal to $\delta_{x + 1, y}$ and $\varphi(m)$ in \cite{Holley70} depends on $x$ in the way $\varphi(m; x) = a_x(1 - q^m)/m$. (In \cite{Holley70} the ZRP is assumed to be spatial-homogeneous, but it is straightforward to generalize the proof to our spatial-inhomogeneous case.) In plain words, if site $x$ is occupied by $m < \infty$ particles, a particle jumps out of site $x$ at rate $a_x(1 - q^m)$, which is the $\Gamma(m)$ defined in \cite[Section 2]{Holley70}, a particle jumps out of site $x$ to site $x + 1$, and If site $x$ is occupied by infinitely many particles, the rate for one particle to jump out of $x$ to $x + 1$ is naturally defined, corresponding to the $\Gamma(\infty)$ defined in \cite[Section 2]{Holley70}, as $a_x(1 - q^{\infty}) = a_x$. In all $q$-TAZRP models considered in our paper, the infinite particles are labelled by consecutive integers. We assume that at any time $t \geq 0$, among all particles at site $x$, particles are stacked vertically with $x_m$ above $x_n$ if $m < n$, and further assume that the particle on the top (that with the smallest label) exists. Furthermore, we assume that this top particle jumps out of $x$ to $x + 1$, and stays below all existing particles there. Hence at all time $t \geq 0$, $x_m(t) \geq x_n(t)$ if $m < n$. The $q$-TAZRP with step initial condition, which is the infinite particle model analyzed in this paper, is defined in the following subsection.

\subsection{Relationship between $q$-TASEP and $q$-TAZRP, and the height function} \label{subsec:TASEP_and_TAZRP}

In general, a ZRP on $\intZ$ has a particle-spacing duality with a simple exclusion process on $\intZ$ \cite{Spitzer70}. In particular, the $q$-TAZRP considered in our paper is dual to the $q$-TASEP introduced by Borodin and Corwin \cite{Borodin-Corwin13}. The intuitive meaning of this duality is simple: we interpret the number of particles at site $k$ in the $q$-TAZRP as the spacing between the $k$-th and $(k - 1)$-th particles in the $q$-TASEP. Below we describe the duality in precise terms and focus on the construction of the $q$-TASEP from the $q$-TAZRP. The inverse construction can be done in the same way, and we omit it since it is not relevant in our paper.

It is better to consider the duality between $q$-TAZRP and $q$-TASEP with infinitely many particles, rather than the finite-particle system defined above in our paper. Suppose that initially no site is occupied by infinitely many particles and  there are infinitely many particles to the left and to the right of $x$, respectively, for each $x \in \mathbb{Z}$. In this case it is impossible to say the ``rightmost'' particle and the ``leftmost'' particle, and in probability $1$ no site is occupied by infinitely many particles at any finite time $t \geq 0$. Let us label the particles at the initial time $t = 0$ as follows. We find the closest occupied site on $(-\infty,0]$ to the origin, and call a particle at this site $x_1$. If there are more than one particle at this site, since we assume that they are vertically stacked and the top particle exists, we call the particle at the top $x_1$. By the ordering method introduced in section \ref{subsec:defn_of_model} we label the other particles at the same position as that of $x_1$ and the particles to the left of $x_1$ as $x_2, x_3, \dotsc$ consecutively, and label the particles to the right of $x_1$ as $x_0, x_{-1}, \dotsc$ consecutively. Hence, at any time $t \geq 0$, we express a state $X(t)=  (\dotsc, x_{-1}(t), x_0(t), x_1(t), x_2(t), x_3(t), \dotsc)$ as a bi-infinite sequence whose terms are in weakly decreasing order, and $x_{n}(t) \to \mp\infty$ as $n \to \pm\infty$. 

There is an alternative unique parametrization $(n_k(t))_{k \in \intZ}$ with $n_k(t) = \text{\#(particles}$ at site $k$ at $\text{time $t$)} < \infty$ and $\sum^{\infty}_{k = -\infty} n_k(t) = \infty$. Let $X(t)$ and $(n_k(t))_{k \in \intZ}$  denote the state of the particles in the $q$-TAZRP described above. Then we define a bi-infinite strictly decreasing sequence $Y(t) = (y_k(t))_{k \in \intZ}$ recursively by
\begin{equation}\label{305-227}
  \begin{split}
    y_{k - 1}(t) - y_k(t) = n_k(t) + 1, \quad y_0(0) = k_0 \quad \text{for some} \quad k_0 \in \intZ,
\end{split}
\end{equation}
and for $t > 0$
\begin{equation} \label{eq:pos_y_0(t)}
  y_0(t) = {k}_0 + \text{$\#$(particles in the $q$-TAZRP $X(t)$ that move from $0$ to $1$ during time $(0, t)$)}.
\end{equation}
If we interpret $y_k(t)$ as the position of the particle labelled $k$ in an infinite particle system on $\intZ$ at time $t$, the dynamics of this particle system can be described as follows: Each particle has an exponential clock that rings independently such that the clock for the  particle labelled $k$ has rate
\begin{equation*}
a_{k} (1 - q^{n_{k}(t)}) = a_k (1 - q^{y_{k - 1}(t) - y_k(t) - 1}).
\end{equation*}
 When the clock rings, the particle moves one step to the right, see Figure \ref{fig:ZRP_SEP_infinite}. (If the right neighbour site is blocked by another particle, we have that the exponential clock rings at rate $0$, and the particle does not move.) This is exactly the $q$-TASEP model defined in \cite{Borodin-Corwin-Sasamoto14} with infinitely many particles without the leftmost paritcle or the rightmost particle. (In \cite{Borodin-Corwin-Sasamoto14}, the rightmost particle exists, though the particle number can be infinity.)

\begin{figure}[htb]
  \centering
  \includegraphics{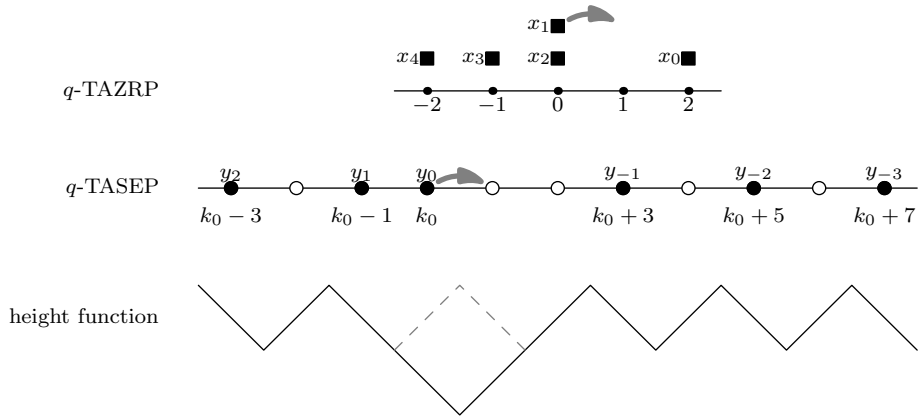}
  \caption{A snapshot of the $q$-TAZRP and the corresponding $q$-TASEP at time $0$. At the bottom is the graph of the height function associated to the $q$-TASEP. In the $q$-TAZRP, the particle on the top of the stack at position $0$ is about to jump, and the corresponding particle $y_0$ jumps in the $q$-TASEP. The jump of $y_0$ changes the height function by flipping the lower $\vee$ corner to the dashed $\wedge$ wedge. Hence if we imagine the height function is for a surface, it grows upward when the $q$-TASEP evolves.}
  \label{fig:ZRP_SEP_infinite}
\end{figure}

If the $q$-TAZRP model has infinitely many particles and the rightmost particle, which is denoted by $x_1(t)$, exists, then $X(t) = (x_k(t))_{k \in \intZ_+} = (x_1(t), x_2(t), \dotsc)$ is uniquely described by $(n_k(t))_{k \in \intZ}$. We also assume that $n_k(t) < \infty$ for all $k \in \intZ$ and $t \geq 0$  as before. The corresponding $q$-TASEP is described by \eqref{305-227} and \eqref{eq:pos_y_0(t)}, and has the feature that sites of $\mathbb{Z}$ far to the right are densely packed with particles at any given time $t$, since $n_k(t) = 0$ for $k \gg 0$. Moreover, if there are  $N < \infty$ particles  in the $q$-TAZRP so that we may denote by $x_1(t)$ and by $x_N(t)$ the positions of the rightmost particle and the leftmost particle, respectively, then again the corresponding $q$-TASEP is described by \eqref{305-227} and \eqref{eq:pos_y_0(t)}. In this case the corresponding $q$-TASEP has infinitely many particles, and sites of $\mathbb{Z}$ far to either side are densely packed with particles at any given time $t$.

We note that in all the forms of the correspondence between $q$-TAZRP and $q$-TASEP described in the above, the  number of particles in the $q$-TASEP is infinite, and we cannot specify  the rightmost particle and the leftmost  particle. But, if in the $q$-TAZRP a unique site, say site $0$, is occupied by infinitely many particles and all sites to its left are empty, then the corresponding $q$-TASEP has the rightmost particle. This correspondence is realized by the $q$-TAZRP with step initial condition, and we explain it below.

Suppose $Y(t)$ gives a $q$-TASEP model. Let the so-called height function $h(x; t)$ associated to $Y(t)$ be a continuous and piecewise linear function in $x$ such that it takes integer values at $k + 1/2$ for $k \in \intZ$, and linearly interpolate between $k - 1/2$ and $k + 1/2$. $h(x; t)$ is defined uniquely by $Y(t)$ that
\begin{equation}
  h(-1/2; t) = \#(\text{particles that passes from $-1$ to $0$ between time $0$ and $t$ in the $q$-TASEP}),
\end{equation}
and for any integer $k$ we have
\begin{equation}
  h(k + 1/2) - h(k - 1/2) =
  \begin{cases}
    -1 & \text{if site $k$ is occupied at time $t$ in the $q$-TASEP}, \\
    1 & \text{if site $k$ is empty at time $t$ in the $q$-TASEP}.
  \end{cases}
\end{equation}
The graph $y = h(x; t)$ in the $xy$ plane is a Dyck path, see Figure \ref{fig:ZRP_SEP_infinite}. It is clear that the Dyck path has a 1-1 correspondence with the $q$-TASEP model, and as time $t$ increases, the Dyck path monotonically moves upward. Actually the move of the Dyck path is in the KPZ universality class, and has been studied by \cite{Borodin-Corwin13}, \cite{Borodin-Corwin-Sasamoto14}, \cite{Ferrari-Veto15}, \cite{Barraquand15} and \cite{Imamura-Sasamoto17}.

A special state of the $q$-TASEP, which can be realized only as an initial condition at time $t = 0$, is that infinitely many particles are labelled as $y_0, y_1, y_2, \dotsc$ such that $y_{i+1} = y_i-1$. The Dyck path corresponding to this configuration of $q$-TASEP has the shape like a step, so this configuration is called the \emph{step initial condition} for $q$-TASEP. Since this configuration of $q$-TASEP is constructed by the configuration of $q$-TAZRP that consists particles $x_1, x_2, x_3, \dotsc$ with $x_k = 0$ for all $k$, we call this configuration also the \emph{step initial condition} for $q$-TAZRP, especially if it is at the initial time $t = 0$.

In the $q$-TASEP with the step initial condition, the generating function $\mathbb{E}(q^{y_m(t)})$ of the $m$-th rightmost particle's position where the probability distribution of $y_m(t)$ is encoded was obtained in  \cite{Borodin-Corwin13} and \cite{Borodin-Corwin-Sasamoto14}. Furthermore, it was shown that an $e_q$-Laplace transform of $q^{y_m(t)}$ is the Fredholm determinant of an integral operator so that the distribution of $y_m(t)$ is obtained from the inversion formula of the $e_q$-Laplace transform. In the $q$-TAZRP, we directly obtain the probability distribution of $x_m(t)$ based on Wang and Waugh's transition probability \cite{Wang-Waugh16} as in Tracy and Widom's approach in \cite{Tracy-Widom08} and \cite{Tracy-Widom08a}.
	
In the end of this subsection we consider the relation between the distribution of particles in the $q$-TAZRP with step initial condition $x_k(0) = 0$ for all $k \geq 1$, and the distribution of particles in the corresponding $q$-TASEP with step initial condition $y_k(0) = -k$ for all $k \geq 0$. At any time $t > 0$, we have that for all integers $x \geq 0$,
\begin{equation}
  \sum^{\infty}_{k = x + 1} n_k(t) = y_x(t) + x.
\end{equation}
Since $x_m(t) > x$ is equivalent to $\sum^{\infty}_{k = x + 1} n_k(t) \geq m$, we have that $x_m(t) > x$ if and only if $y_x(t) + x \geq m$. We conclude that
\begin{equation} \label{eq:relation_q_TASEP_q_TAZRP_prob_distr}
  \Prob_{0^{\infty}}(x_m(t) > x) = \Prob^{\qTASEP}_{\step}(y_x(t) \geq m - x),
\end{equation}
where the two sides are the probabilities that $x_m(t) > x$ and $y_x(t) \geq m - x$ respectively in the $q$-TAZRP/$q$-TASEP with step initial condition given above.

\subsection{Main results}

First we state the result on the distribution function for a single particle $x_n(t)$ in the $N$-particle $q$-TAZRP, under the initial condition that $ Y = (y_1, \dotsc, y_N)$. This result is based on transition probability $P_Y(X;t)$ obtained by Wang and Waugh \cite{Wang-Waugh16}, which is in turn a generalization of the result by Korhonen and Lee \cite{Korhonen-Lee14} in the homogeneous case. We first define notations to be used later in this paper.

First, throughout this paper, we let
\begin{equation}\label{437-34}
  b_x = (1 - q)a_x
\end{equation}
for all $x \in \intZ$. For a state $X = (x_1, \dotsc, x_N)$ expressed in terms of $(n_k)_{k \in \intZ}$, let
\begin{equation} \label{eq:defn_W(X)}
  W(X) = \prod_{k \in \intZ}[n_k]_q!,
\end{equation}
where $[m]_q!$ is the $q$-deformed factorial, defined by
\begin{equation}
  [m]_q! = [1]_q[2]_q\cdots [m]_q = \prod^m_{k = 1} \frac{1-q^{k}}{1-q}.
\end{equation}
Note that all but finitely many $n_k$ are $0$, and the factor $[0]_q! = 1$, so the infinite product in \eqref{eq:defn_W(X)} is well defined.

For a permutation $\sigma \in S_N$, an \emph{inversion} of $\sigma$ is an ordered pair $(\sigma(i),\sigma(j))$ such that $i<j$ and $\sigma(i)>\sigma(j)$. Let
 \begin{equation}
 S(w_{\alpha},w_{\beta}) = -\frac{qw_{\beta} - w_{\alpha}}{qw_{\alpha} - w_{\beta}}
 \end{equation}
where $w_{\alpha}$ and $w_{\beta}$ are complex variabls and
\begin{equation} \label{A-sigma}
  A_{\sigma}(w_1,\dotsc, w_N) = \prod_{\text{$(\beta,\alpha)$ is an inversion of $\sigma$}} S(w_{\alpha},w_{\beta}).
\end{equation}
We define $A_{\sigma}(w_1,\dotsc, w_N) = 1$ if $\sigma$ is the identity permutation. For notational simplicity we define $\prod^{'}$ as an extension of $\prod$. If $f(k)$ is defined for all $k \in \intZ$, then
\begin{equation}
  \sideset{}{'}\prod_{k=m}^n f(k) =
  \begin{cases}
    \displaystyle\prod_{k=m}^nf(k) & \text{if $n \geq m$},\\
    1 & \text{if $n = m - 1$},\\
    \displaystyle\prod_{k=n+1}^{m-1}\frac{1}{f(k)} & \text{if $n < m - 1$}.
  \end{cases}
\end{equation}
The integral sign $\dashint$ is a shorthand for $\frac{1}{2 \pi i} \oint$.
\begin{proposition}\cite{Wang-Waugh16}\label{Wang-Waugh}
  Given the notations in the above, the transition probability of the $q$-TAZRP with rates \eqref{eq-rate} is given by
  \begin{multline}
    P_Y(X;t) = \frac{1}{W(X)} \prod^N_{k=1} \frac{-1}{b_{x_k}} \sum_{\sigma \in S_N} \dashint_{C} dw_1 \cdots \dashint_{C} dw_N A_{\sigma}(w_1, \dotsc, w_N) \\
    \times \prod^N_{j=1} \left[ \prodprime^{x_j}_{k=y_{\sigma(j)}} \left( \frac{b_k}{b_k - w_{\sigma(j)}} \right) e^{-w_jt} \right],
  \end{multline}
  where $C$ is a counter-clockwise circle centered at $0$ with sufficiently large radius enclosing all singularities $b_k$.
\end{proposition}

The first main result of this paper is the marginal distribution for the $n$-th particle's position at time $t > 0$. We provide the probability in two forms which are reminiscent of Tracy and Widom's results in \cite[Theorem 5.1 and 5.2]{Tracy-Widom08} for the ASEP.   First define for any $r$ variables $w_1, \dotsc, w_r$
\begin{equation}\label{433-34}
  B_r(w_1, \dotsc, w_r) = \prod_{1 \leq i < j \leq r} \frac{w_i - w_j}{qw_i - w_j}.
\end{equation}
Then for the subset $S=\{s_1,\dotsc, s_r\} \subseteq \{1,2,\dotsc, N\}$ with $\lvert S \rvert = r$, we define for $n \leq r$
\begin{equation}
  \begin{aligned}
    c_S(n) = {}& (-1)^n q^{\frac{n(n - 1)}{2} - nr + \sum^r_{i = 1} s_i} \qbinom{r - 1}{n - 1}, \\
    \tilde{c}_S(n) = {}& (-1)^n q^{\frac{n(n - 1)}{2} + \frac{r(r - 1)}{2} + \sum^r_{i = 1} s_i} \qbinom{r - 1}{n - 1},
  \end{aligned}
\end{equation}
where $q$-binomial coefficient $\qbinom{n}{k}$ is defined by
\begin{equation}
  \qbinom{n}{k} = \frac{[n]_q!}{[k]_q![n-k]_q!}.
\end{equation}
Let
\begin{equation}
  I(w_{s_1}, w_{s_2}, \dotsc, w_{s_r};M, t) = \prod_{1 \leq i < j \leq r} B_r(w_{s_1}, \dotsc, w_{s_r}) \prod^r_{i=1} \left[ \prodprime^M_{k = y_{s_i}} \left( \frac{b_k}{b_k - w_{s_i}} \right) e^{-w_{s_i}t} \right].
\end{equation}
\begin{theorem} \label{main-theorem-1}
  Given the notations in the above,
  \begin{multline} \label{eq:distr_X_m}
    \Prob_Y(x_n(t) > M) = \sum^N_{r = n} (-1)^r \\
    \times \sum_{S = \{ s_1 , \dotsc, s_r \} \subseteq \{ 1, \dotsc, N \} \text{ and } \lvert S \rvert = r} (-1)^r c_S(n) \dashint_{{C}}\frac{dw_{s_1}}{w_{s_1}}\cdots \dashint_{{C}}\frac{dw_{s_r}}{w_{s_r}}~I(w_{s_1}, \dotsc, w_{s_r};M, t),
  \end{multline}
  where the contour $C$ is a positively oriented circle centred at $0$ and enclosing all $b_k$.
  \begin{multline} \label{eq:distr_X_N-m+1}
    \Prob_Y(x_{N - n + 1}(t)\leq M) = \sum^N_{r = n} q^{-rN} \\
    \times \sum_{S = \{ s_1 , \dotsc, s_r \} \subseteq \{ 1, \dotsc, N \} \text{ and } \lvert S \rvert = r} \tilde{c}_S(n) \dashint_{\tilde{C}_1} \frac{dw_{s_1}}{w_{s_1}} \cdots \dashint_{\tilde{C}_r} \frac{dw_{s_r}}{w_{s_r}}~I(w_{s_1}, \dotsc, w_{s_r};M, t),
  \end{multline}
  where the contour $\tilde{C}_j$ is a positively oriented simple closed curve which encloses all $b_k$ but does not enclose $0$, and $\tilde{C}_j$ encloses $q\tilde{C}_i$ if $j>i$.
\end{theorem}
We note that the $n = N$ case of \eqref{eq:distr_X_N-m+1}, which is stated separately in Proposition \ref{prop:rightmost}, is similar to the $q$-monent formulas for the $q$-TASEP and more generally the $q$-Whittaker processes, see \cite[Proposition 3.5]{Borodin-Corwin13} and \cite[Theorem 2.11]{Borodin-Corwin-Sasamoto14}. The $n = N$ case of \eqref{eq:distr_X_m}, which is stated separately in Proposition \ref{prop:leftmost}, generalizes \cite[Theorem 3.1]{Korhonen-Lee14} that is the homogeneous case of Proposition \ref{prop:leftmost}.

If the initial condition $Y$ is  $Y = 0^N = (0, \dotsc, 0)$, then it is clear that the formulas \eqref{eq:distr_X_m} and \eqref{eq:distr_X_N-m+1} can be simplified. Moreover, if we consider the limiting case of $0^N = (0, \dotsc, 0)$, that is, the step initial condition, $Y = (y_1, y_2, \dotsc) = 0^{\infty} := (0, 0, \dotsc)$, then the probability $\Prob_{0^{\infty}}(x_m(t) > M)$ can be expressed by a contour integral of a Fredholm determinant. To this end, we define the integral operator $K_{M, t}$ from $L^2(\Gamma)$ to $L^2(\Gamma)$ where $\Gamma$ is a large enough contour enclosing $0$, such that
\begin{equation} \label{eq:op_K_Mt}
  (K_{M,t}f)(w) = \frac{1}{2\pi i} \oint_{\Gamma} K_{M,t}(w,w')f(w')~dw'
\end{equation}
with kernel
\begin{equation} \label{eq:kernel_K_Mt}
  K_{M,t}(w,w') = \frac{e^{-wt}}{qw'-w}\prod_{k=0}^M \left( \frac{b_k}{b_k - w} \right)
\end{equation}
whose poles $b_0, b_1, \dotsc, b_M$ are enclosed in $\Gamma$, and the scaled down contour $q\Gamma$ is enclosed in $\Gamma$. It is clear that $K_{M, t}$ is a trace class operator and the Fredholm determinant $\det(I + \zeta K_{M, t})$ is well defined for all $\zeta \in \compC$. Here we do not need to specify the shape of $\Gamma$, since later we only need the Fredholm determinant $\det(I + \zeta K_{M, t})$ that is invariant when $\Gamma$ is continuously deformed, as shown in \cite[Proposition 1]{Tracy-Widom09}.

We note that the integral operator $K_{M, t}$ appears also in the study of $q$-TASEP in the so-called Cauchy-type determinant, see \cite[Proposition 3.10]{Borodin-Corwin-Sasamoto14}, and see \cite[Theorem 3.23]{Borodin-Corwin13} for an analogous ``large contour formula'' for the $q$-Whittaker processes.
\begin{theorem} \label{theorem-step}
For the $q$-TAZRP,
\begin{equation} \label{theorem-1}
  \Prob_{0^{\infty}}(x_m(t) > M) = \frac{1}{2\pi i} \oint_C \frac{d\zeta}{\zeta} \frac{\det(I + \zeta K_{M,t})}{(1- \zeta)(1-q\zeta)\cdots (1-q^{m-1}\zeta)},
\end{equation}
where $C$ is a positively oriented circle centred at $0$ large enough to enclose that all singularities $1, q^{-1}, \dotsc, q^{1 - m}$.
\end{theorem}
Here we remark that the probability $\Prob_{0^{\infty}}(x_m(t) > M)$ depends on the values of $b_k$ with $0 \leq k \leq M$, but not the order of them. This symmetry is not obvious from the definition of the model.

Below we give the asymptotics of the probability $\Prob_{0^{\infty}}(x_m(t) > M)$ as  $t \to \infty$ when $m$ is fixed and $M$ is scaled accordingly. We consider only a special case that all but a fixed number of $b_k$ are equal. Since $\Prob_{0^{\infty}}(x_m(t) > M)$ is invariant if $b_k$ and $b_j$ ($0 \leq k, j \leq M$) are swapped, we may assume the spiked $b_k$ are $b_0, b_1, \dotsc, b_l$. This setting is analogous to the $q$-TASEP model with all but a fixed number of particles having the same speed parameter, as considered in \cite{Barraquand15}.
\begin{theorem} \label{thm:asymp}
  Suppose
  \begin{equation} \label{eq:parameter_specification}
    M = n + l + 1, \quad t = n - \tau \sqrt{n} \quad b_k = \frac{\beta_k}{\sqrt{n}} \quad \text{for $k = 0, 1, \dotsc, l$} \quad \text{and} \quad b_k = 1 \quad \text{for all $k > l$}.
  \end{equation}
 Let $l, \beta_0, \dotsc, \beta_l, \tau$ be fixed. Then, as $n \to \infty$, we have uniformly for $\zeta$ in any compact subset of $\compC$,
  \begin{equation} \label{eq:uniform_conv_K_to_K}
    \lim_{n \to \infty} \left. \det(1 + \zeta K_{M, t}) \right\rvert_{M = n + l + 1 \text{ and } t = n - \tau \sqrt{n}} = \det(I + \zeta \K),
  \end{equation}
  where $K_{M, t}$ is defined in \eqref{eq:op_K_Mt} and \eqref{eq:kernel_K_Mt} with parameters specified in \eqref{eq:parameter_specification}, and $\K$ is the integral operator on the infinite vertical contour $\Gamma_{\infty} = \{ -1 - yi \mid y \in \realR \}$, oriented downward, with kernel
  \begin{equation} \label{eq:defn_K(z, w)}
    \K(z, w) = \K_{\tau; \beta_0, \dotsc, \beta_l}(z, w) =  \frac{e^{w^2/2 + \tau w}}{w - qz} \gamma(w), \quad \text{where} \quad \gamma(w) = \prod^l_{k = 0} \frac{\beta_k}{\beta_k - w}.
  \end{equation}
  Hence
  \begin{equation} \label{eq:convergence_of_contour_integral_in_det_K}
    \lim_{n \to \infty} \left. \Prob_{0^{\infty}}(x_m(t) > M) \right\rvert_{M = n + l + 1 \text{ and } t = n - \tau \sqrt{n}} = \frac{1}{2\pi i} \oint_C \frac{d\zeta}{\zeta} \frac{\det(I + \zeta \K)}{(1- \zeta)(1-q\zeta)\cdots (1-q^{m-1}\zeta)},
  \end{equation}
  where $C$ is the same as in Theorem \ref{theorem-step}.
\end{theorem}

\begin{remark} \label{rmk:specialization}
  To include the important special case that all $b_k$ are identically $1$ into Theorem \ref{thm:asymp}, we simply let $l = -1$, and then $\gamma(w) = 1$. In this case, we denote the limiting kernel $\K(z, w) = \K_{\tau; -}(z, w)$.
\end{remark}

Here we remark that unlike the dual $q$-TASEP process with step initial condition in which the limiting distribution of $x_m(t)$ is simple for small $m$, the limiting distribution of $x_m(t)$ is non-trivial even for small $m$ in the $q$-TAZRP with step initial condition. For example, by applying the residue theorem to \eqref{eq:convergence_of_contour_integral_in_det_K}, we have
\begin{equation} \label{eq:must_be_labelled}
  \lim_{n \to \infty} \Prob_{0^{\infty}}(x_1(t) \leq M) = \lim_{n \to \infty} \left.  \det(I + K_{M,t}) \right\rvert_{M = n + l + 1 \text{ and } t = n - \tau \sqrt{n}} = \det(I + \K).
\end{equation}
It shows that the limiting distribution of $x_1(t)$ is non-Gaussian, on the other hand the right-most particle in the $q$-TASEP behaves as a free particle and hence the fluctuation of the first particle's position is Gaussian at large time.

In the study of $1$-dimensional interacting particle systems like the Totally Asymmetric Exclusion Process (TASEP) and its variations, usually the focus is on the behaviour of particles in the bulk rather than those at the edge, partly because the limiting behaviour of bulk particles often shows the feature of the Kadar--Parisi--Zhang (KPZ) universality class, partly because the limiting behaviour of the edge particles is often trivial. An exception is the recent paper on the facilitated TASEP by Baik, Barraquand, Corwin and Suidan \cite{Baik-Barraquand-Corwin-Suidan16}, but the limiting behaviour of edge particles in their model is quite different from that in Theorem \ref{theorem-step}. Also nontrivial limiting behaviours of the edge particles in the ASEP model are found by Tracy and Widom \cite{Tracy-Widom09}, \cite{Tracy-Widom09a}. It turns out that the limiting distributions for the rightmost particles in the $q$-TAZRP with the step initial condition are the same as those for the leftmost particles in the ASEP with step initial condition given in \cite[Theorem 2]{Tracy-Widom09} when all $b_k$ are identical. Hence the limiting behaviour of the edge particles in the $q$-TAZRP and ASEP may be in a new universal class and worth more investigation.

\begin{corollary} \label{cor:TW_formula}
  \begin{enumerate}[label=(\alph*)]
  \item \label{enu:cor_TW_formula_a}
    The specialization of $\K$ in Theorem \ref{thm:asymp} with $l = -1$ as in Remark \ref{rmk:specialization} satisfies for all $\zeta \in \compC$,
    \begin{equation} \label{eq:Mehler_kernel}
      \det(I + \zeta \K_{\tau; -}) = \det(I + \zeta \hat{K} \chi_{(\tau(1 + q)/(1 - q), \infty)}),
    \end{equation}
    where $\hat{K}$ is the operator on $L^2(\realR)$, defined in \cite[Theorem 2]{Tracy-Widom09} with the parameters $p$ and $q$ replaced by $q/(1 + q)$ and $1/(1 + q)$ respectively, and has the kernel
    \begin{equation} \label{eq:defn_hat_K}
      \hat{K}(z, z') = \frac{1}{\sqrt{2\pi}(1 + q)} e^{-\frac{1 + q^2}{(1 + q)^2} \frac{z^2 + z'^2}{4} + \frac{q}{(1 + q)^2} zz'}.
    \end{equation}
  \item \label{enu:cor_TW_formula_b}
    The limiting distribution of the $m$-th rightmost particle in our $q$-TAZRP with step initial condition and all $b_k = 1$ is the same as the limiting distribution of the $m$-th leftmost particle in the ASEP with step initial condition and the left and right jumping rates equal to $q/(1 + q)$ and $1/(1 + q)$ respectively, under proper scaling.
  \end{enumerate}
\end{corollary}

From the exact contour integral formula \eqref{theorem-1}, we can also derive the limiting distribution of $x_m(t)$ when both $m$ and $t$ approach $\infty$ and $t/m$ is a constant. The asymptotic analysis is similar to that for \cite[Theorem 3]{Tracy-Widom09}, but can be simplified, due to the reason mentioned in Remark \ref{rmk:technical}. However, we omit the asymptotic analysis, since the limit of $\Prob_{0^{\infty}}(x_m(t) > M)$ can be expressed directly as a corollary of \eqref{eq:relation_q_TASEP_q_TAZRP_prob_distr} and the results in \cite{Barraquand15} and \cite{Ferrari-Veto15}. See Appendix \ref{sec:TW_limit} for details.

\subsection*{Organization of the paper}

In Section \ref{sec:algebraic} we prove Theorems \ref{main-theorem-1} and \ref{theorem-step}, and in Section \ref{sec:asymptotic} we prove Theorem \ref{thm:asymp} and Corollary \ref{cor:TW_formula}. In Appendix \ref{sec:TW_limit} we discuss the limiting behaviour of particles in the bulk based on known results on $q$-TASEP.

\subsection*{Acknowledgements}

We thank the helpful discussion with Guillaume Barraquand.

\section{Distribution of the position of a tagged particle} \label{sec:algebraic}

\subsection{Finite system with an arbitrary initial condition $Y$}

Let us consider an $N$-particle system with an arbitrary initial condition $X(0) = (x_1(0), \dotsc, x_N(0)) = Y = (y_1, \dotsc, y_N)\in \Weylchamber$.

 We find the distribution functions of $x_n(t)$ and $x_{N - n + 1}(t)$ given by $n$-fold contour integrals. We first solve the problem for $n = N$ (the leftmost particle) and $N - n + 1 = 1$ (the rightmost particle), and then solve for general $n$ based on the two special cases.

In this section, we need the following notations and identities. First, recall the notations of $A_{\sigma}$ defined in \eqref{A-sigma} and $B_n$ defined in \eqref{433-34}. We have a formula that is proved in \cite[Formula (5.1)]{Wang-Waugh16}:
\begin{equation} \label{eq:basic_id}
  \sum_{\sigma \in S_n} A_{\sigma}(w_1, \dotsc, w_n) = [n]_q! B_n(w_1, \dotsc, w_n).
\end{equation}
Also later we use the elementary symmetric function
\begin{equation} \label{eq:elem_symm}
  e_l(w_1, \dotsc, w_n) = \sum_{1 \leq i_1 < i_2 < \dotsb < i_l \leq n} w_{i_1} w_{i_2} \dotsm w_{i_l}.
\end{equation}
and the Vandermonde determinant
\begin{equation} \label{eq:Vandermonde}
  \Delta_n(w_1, \dotsc, w_n) = \det(w^{k - 1}_j)^n_{j, k = 1} = \prod_{1 \leq i < j \leq n} (w_j - w_i).
\end{equation}

\subsubsection{Positions of the leftmost particle and the rightmost particle}

\paragraph{Leftmost particle}

The distribution of $x_N(t)$ is given by the following proposition:
\begin{proposition} \label{prop:leftmost}
  For the $N$-particle $q$-TAZRP model with initial condition $Y = (y_1, \dotsc, y_N)$, the leftmost particle $x_N$ satisfies
  \begin{equation}
    \mathbb{P}_Y(x_N(t) >M) = \int_{C}\frac{dw_1}{w_1}\cdots \int_{C}\frac{dw_N}{w_N} B_N(w_1, \dotsc, w_N) \prod_{j=1}^N \left[ \prodprime_{k=y_{j}}^M \left( \frac{b_k}{b_k - w_{j}} \right) e^{-w_j t} \right],
  \end{equation}
  where the contour $C$ is specified as in Theorem \ref{main-theorem-1}.
\end{proposition}
The proof of Proposition \ref{prop:leftmost} relies on the following technical result:
\begin{proposition}\label{thm:left_most}
  Let $q \in (0,1)$, $Y = (y_1, \dotsc, y_N) \in \intZ^N$, $b_k$ be in a compact subset of $(0,\infty)$ for all $k \in \mathbb{Z}$, and $F(w_1, \dotsc, w_N)$ be a meromorphic function in $w_1, \dotsc, w_N$. Let $C$ be a positively oriented, large enough circle centered at $0$ and enclosing all $b_k$, and $F(w_1, \dotsc, w_N)$ be continuous if $w_k \in C$. Denoting
  \begin{multline} \label{eq:defn_I(x)}
    I(x_1, \dotsc, x_N) = \frac{1}{W(X)} \left( \prod_{k=1}^N\frac{-1}{b_{x_k}} \right) \dashint_{C} dw_1 \dotsi \dashint_{C} dw_N F(w_1, \dotsc, w_N) \\
    \times \sum_{\sigma \in S_N} A_{\sigma}(w_1, \dotsc, w_N) \left[ \prod_{j=1}^N \left( \prodprime_{k=y_{\sigma(j)}}^{x_j} \frac{b_k}{b_k - w_{\sigma(j)}} \right) e^{-w_j t} \right]
  \end{multline}
  for $X = (x_1, \dotsc, x_N) \in \Weylchamber \subseteq \intZ$, where $A_{\sigma}$ is defined in \eqref{A-sigma}, then we have
  \begin{multline} \label{eq-identity}
    \sum_{X = (x_1, \dotsc, x_N) \text{ with } M < x_N \leq \dotsb \leq x_1} I(x_1, \dotsc, x_N) \\
    = \dashint_{C}dw_1\cdots \dashint_{C}dw_N F(w_1, \dotsc, w_N) B_N(w_1, \dotsc, w_N) \prod_{j=1}^N \left[ \frac{e^{-w_j t}}{w_j}\prodprime_{k=y_{j}}^{M}\frac{b_k}{b_k - w_{j}} \right].
  \end{multline}
\end{proposition}

Using the expression
\begin{equation} \label{eq-marginal}
  \mathbb{P}_Y(x_N(t) > M) = \sum_{X = (x_1,\dotsc, x_N) \text{ with } M<x_N \leq \dotsb \leq x_1} P_Y(X;t)
\end{equation}
and the contour integral formula  for $P_Y(X; t)$ in Proposition \ref{Wang-Waugh}, we have that Proposition \ref{prop:leftmost} is the special case of Proposition \ref{thm:left_most} with $F(w_1, \dotsc, w_N) = 1$.

\begin{remark}
The identity \eqref{eq-identity} is a generalization of \cite[Formula (3.11)]{Korhonen-Lee14}. The symmetrization identity in the $q$-Hahn TASEP \cite[Corollary 5.22]{Borodin-Corwin-Petrov-Sasamoto15}) with two parameters is also a generalization of \cite[Formula (3.11)]{Korhonen-Lee14} in a different direction. Similar to \cite[Formula (3.11)]{Korhonen-Lee14}, the proof of \eqref{eq-identity} is by a straightforward mathematical induction, while the proof of \cite[Corollary 5.22]{Borodin-Corwin-Petrov-Sasamoto15} is based on a result of the spectral Plancherel theorem, see \cite{Borodin-Corwin-Petrov-Sasamoto15} for more details.
\end{remark}

\begin{proof}[Proof of Proposition \ref{thm:left_most}]
  We prove \eqref{eq-identity} by induction on $N$. The $N = 1$ case of \eqref{eq-identity} is reduced to
  \begin{equation} \label{eq:n=1_left_most}
    \sum^{\infty}_{x_1 = M+1}\dashint_{C}  \frac{-1}{b_{x_1}} F(w_1)e^{-w_1t}\sideset{}{'}\prod^{x_1}_{k = y_1} \frac{b_k}{b_k - w_1} ~dw_1 = \dashint_{C}\frac{e^{-w_1t}}{w_1}F(w_1) \sideset{}{'}\prod^{M}_{k = y_1} \frac{b_k}{b_k - w_1}~dw_1.
  \end{equation}
  Suppose without loss of generality that $M\geq y_1$. Since the radius of $C$ is  large enough and all $b_k$ are in a compact subset of $(0,\infty)$,
  \begin{equation}
    \sum^{\infty}_{x_1 = M + 1} \frac{-1}{b_{x_1}} F(w_1)\prod^{x_1}_{k = y_1} \frac{b_k}{b_k - w_1}
  \end{equation}
  converges on $C$ uniformly. Hence it suffices to show that
  \begin{equation} \label{eq1-lemma}
    \sum^{\infty}_{x_1 = M + 1} \frac{-1}{b_{x_1}}\prod^{x_1}_{k = M+1} \frac{b_k}{b_k - w_1} = \frac{1}{w_1}.
  \end{equation}
  This is done by the telescoping trick, noting that
  \begin{equation} \label{eq2-lemma}
    w_1 \left( \frac{-1}{b_{x_1}} \prod^{x_1}_{k = M + 1} \frac{b_k}{b_k - w_1} \right) = -\prod^{x_1}_{k = M + 1} \frac{b_k}{b_k - w_1} + \prod^{x_1 - 1}_{k = M + 1} \frac{b_k}{b_k - w_1}.
  \end{equation}

  Now we consider the general $N$ case inductively, assuming that \eqref{eq-identity} holds for $1, 2, \dotsc, N - 1$. We write the left-hand side of \eqref{eq-identity} as
  \begin{equation} \label{eq:1135-219}
    \sum^{\infty}_{m = M + 1} \sum_{l=1}^{N} I^{(l)}(m), \quad \text{where} \quad I^{(l)}(m) = \sum_{m<x_{N-l} \leq \dotsb \leq x_1} I(x_1,\dotsc, x_{N - l}, \underbrace{m, \dotsc, m}_{x_N =\dotsb = x_{N-l+1}=m}).
  \end{equation}
  Then with the index $U = (u_1, \dotsc, u_{N - l}) \in \Weylchamber[N - l]$,
  \begin{equation} \label{eq:1235-219}
    \begin{split}
      I^{(l)}(m) = {}& \frac{(-1/b_m)^{l}}{[l]_q!}\sum_{ m < u_{N - l} \leq \dotsb \leq u_1} \frac{1}{W(U)} \left( \prod^{N - l}_{k = 1} \frac{-1}{b_{u_k}} \right) \dashint_{C} dw_1 \dotsi \dashint_{C} dw_N ~ F(w_1, \dotsc, w_N) \\
      & \times  \sum_{\sigma \in S_N}A_{\sigma}(w_1, \dotsc, w_N) \prod^N_{j = N - l + 1} \left[ \prodprime^m_{k = y_{\sigma(j)}} \left( \frac{b_k}{b_k - w_{\sigma(j)}} \right) e^{-w_{\sigma(j)} t} \right] \\
      & \times \prod^{N - l}_{j = 1} \left[ \prodprime^{u_j}_{k = y_{\sigma(j)}} \left( \frac{b_k}{b_k - w_{\sigma(j)}} \right) e^{-w_{\sigma(j)} t} \right],
    \end{split}
  \end{equation}
  such that $I^{(N)}(m)$ degenerates to $I(m, \dotsc, m)$ in \eqref{eq:defn_I(x)}. To evaluate $I^{(l)}(m)$, we decompose the index set $S_N$ into $\binom{N}{l}$ disjoint subsets $S_N(I_l)$ where $I_l \subset \{1, 2, \dotsc, N \}$ has $l$ elements and
  \begin{equation} \label{eq:defn_S_n(I)}
    S_N(I_l) = \{ \sigma \in S_N \mid \sigma(\{N - l + 1, \dotsc, N - 1, N \}) = I_l \}.
  \end{equation}
  In the subsequent part of the paper, we take the notational convention that $I_l = \{ i_1, \dotsc, i_l \}$ and $I^c_l = \{ 1, \dotsc, N \} \setminus I_l = \{ j_1, \dotsc, j_{N-l} \}$ such that $i_k$ and $j_k$ are in ascending order. Let us define a bijective mapping  $\varphi_{I_l}(\cdot \mid \cdot) : S_l \times S_{N - l} \longrightarrow S_N(I_l)$ by
  \begin{equation} \label{eq:defn_phi_I(lambda_tau)}
    \varphi_{I_l}(\lambda \mid \tau)(k) =
    \begin{cases}
      j_{\tau(k)} & \text{if $k \leq N - l$}, \\
      i_{\lambda(k + l - N)} & \text{if $k > N - l$}.
    \end{cases}
  \end{equation}
  Then, for each $\sigma \in S_N(I_l)$, there is a unique pair $\lambda \in S_l$ and $\tau \in S_{N-l}$ such that $\varphi_{I_l}(\lambda \mid \tau) = \sigma$ by \eqref{eq:defn_phi_I(lambda_tau)}, and then
  \begin{equation} \label{eq:decomp_A_{sigma}}
    A_{\sigma}(w_1, \dotsc, w_N) = G_{I_l}(w_1, \dotsc, w_N) A_{\lambda}(w_{i_1}, \dotsc, w_{i_l}) A_{\tau}(w_{j_1}, \dotsc, w_{j_{N - l}}),
  \end{equation}
  where
  \begin{equation} \label{eq:defn_G_I_l}
    G_{I_l}(w_1, \dotsc, w_N) = \prod_{\substack{i \in I_l,\ j \in \{ 1, \dotsc, N \} \setminus I_l \\ i < j}} -\frac{qw_j - w_i}{qw_i - w_j}.
  \end{equation}
  Dividing the summation over $S_N$ in \eqref{eq:1235-219} into $\binom{N}{l}$ terms where each term is the summation over an $S_N(I_l)$, and using \eqref{eq:decomp_A_{sigma}} to decompose the $A_{\sigma}$ factor, we have
  \begin{multline} \label{141-220}
    I^{(l)}(m) = \frac{(-1/b_m)^{l}}{[l]_q!}\sum_{I_l \subseteq \{ 1, \dotsc, N \}} \sum_{\lambda \in S_l} \dashint_{C} dw_{i_1} \dotsi \dashint_{C} dw_{i_l} A_{\lambda}(w_{i_1}, \dotsc, w_{i_l}) \\
    \times \prod^l_{r = 1} \left[ \prodprime^m_{k = y_{i_{\lambda(r)}}} \left( \frac{b_k}{b_k - w_{i_{\lambda(r)}}} \right) e^{-w_{i_{\lambda(r)}} t} \right] \\
    \times \left(
      \begin{aligned}
        & \sum_{ m < u_{N - l} \leq \dotsb \leq u_1}  \frac{1}{W(U)} \left( \prod^{N - l}_{k = 1} \frac{-1}{b_{u_k}} \right) \dashint_{C} dw_{j_1} \dotsi \dashint_{C} dw_{j_{N - l}} F(w_1, \dotsc, w_N) \\
        & \times G_{I_l}(w_1, \dotsc, w_N) \sum_{\tau \in S_{N - l}} A_{\tau}(w_{j_1}, \dotsc, w_{j_{N - l}}) \prod^{N - l}_{r = 1}  \prodprime^{u_r}_{k = y_{j_{\tau(r)}}} \left( \frac{b_k}{b_k - w_{j_{\tau(r)}}} \right) e^{-w_{j_{\tau(r)}} t}
      \end{aligned}
    \right).
  \end{multline}
  Applying the induction hypothesis to \eqref{141-220}, we simplify it as
  \begin{multline} \label{eq:I_tilde_a_part}
    I^{(l)}(m) =  \frac{(-1/b_m)^{l}}{[l]_q!}\sum_{I_l \subseteq \{ 1, \dotsc, N \}} \sum_{\lambda \in S_l} \dashint_{C} dw_{i_1} \dotsi \dashint_{C} dw_{i_l} A_{\lambda}(w_{i_1}, \dotsc, w_{i_l}) \\
    \times \prod^l_{r = 1} \left[ \prodprime^m_{k = y_{i_{\lambda(r)}}} \left( \frac{b_k}{b_k - w_{i_{\lambda(r)}}} \right) e^{-w_{i_{\lambda(r)}} t} \right] \\
    \times \left(
      \begin{aligned}
        & \dashint_{C} \frac{dw_{j_1}}{w_{j_1}} \dotsi \dashint_{C} \frac{dw_{j_{N - l}}}{w_{j_{N - l}}} F(w_1, \dotsc, w_N) G_{I_l}(w_1, \dotsc, w_N) \\
        & \times B_{N - l}(w_{j_1}, \dotsc, w_{j_{N - l}}) \prod^{N - l}_{r = 1} \left[ \prodprime^m_{k = y_{j_r}} \left( \frac{b_k}{b_k - w_{j_r}} \right) e^{-w_{j_r} t} \right]
      \end{aligned}
    \right).
  \end{multline}
  Next, by \eqref{eq:basic_id}, we further have
  \begin{equation} \label{eq:I^(l)_m_by_Gamma}
    I^{(l)}(m) = \left( \frac{-1}{b_m} \right)^{l} \dashint_{C} dw_1 \dotsi \dashint_{C} dw_N \Gamma(w_1, \dotsc, w_N; l) F(w_1, \dotsc, w_N) \prod^N_{r = 1} \left[ \prodprime^m_{k = y_r} \left( \frac{b_k}{b_k - w_r} \right) e^{-w_r t} \right],
  \end{equation}
  where
  \begin{equation} \label{eq:defn_Gamma(w;l)}
    \Gamma(w_1, \dotsc, w_N; l) = \sum_{I_l \subseteq \{ 1, \dotsc, N \} } \frac{G_I(w_1, \dotsc, w_N)}{\prod^{N - l}_{k = 1} w_{j_k}} B_l(w_{i_1}, w_{i_2} \dotsc, w_{i_l}) B_{N - l}(w_{j_1}, \dotsc, w_{j_{N - l}}).
  \end{equation}
  Let
  \begin{equation} \label{eq:express_Gamma_m(w)}
    \Gamma_m(w_1, \dotsc, w_N) = \sum^N_{l = 1} \left( \frac{-1}{b_m} \right)^l  \Gamma(w_1, \dotsc, w_N; l).
  \end{equation}
  Since the left-hand side of \eqref{eq-identity} is given by \eqref{eq:1135-219}, we plug \eqref{eq:I^(l)_m_by_Gamma} into \eqref{eq:1135-219}, and have that the left-hand side of \eqref{eq-identity} can be expressed as
  \begin{equation} \label{eq:simplified_I^n_Y}
    \sum^{\infty}_{m = M + 1} \dashint_{C} dw_1 \dotsi \dashint_{C} dw_N ~\Gamma_m(w_1, \dotsc, w_N) \prod^N_{r = 1} \left[ \prodprime^m_{k = y_r} \left( \frac{b_k}{b_k - w_r} \right) e^{-w_r t} \right]  F(w_1, \dotsc, w_N).
  \end{equation}

  To complete the proof, we need to show that \eqref{eq:simplified_I^n_Y} is equal to the right-hand side of \eqref{eq-identity}. To this end, we need to simplify $\Gamma(w_1, \dotsc, w_N; l)$ and $\Gamma_m(w_1, \dotsc, w_N).$ Let us consider
  \begin{multline} \label{eq:to_be_in_matrix_forms}
    q^{(N - l)(N - l - 1)/2} (-1)^{N(N - 1)/2} \prod^N_{k = 1} w_k \prod_{1 \leq i < j \leq N}(qw_i - w_j) \Gamma(w_1, \dotsc, w_N; l) \\
    = \sum_{I_l \subseteq \{ 1, \dotsc, N \} } (-1)^{(i_1 + \dotsb + i_l) - l(l + 1)/2} \prod^l_{k = 1} w_{i_k} \prod_{1 \leq k < r \leq l} (w_{i_r} - w_{i_k}) \prod_{1 \leq k < r \leq N - l} (qw_{j_r} - qw_{j_k}) \\
    \times \prod^l_{k = 1} \prod^{N - l}_{r = 1} (qw_{j_r} - w_{i_k}).
  \end{multline}
  For each $I_l$, the summand on the right-hand side of \eqref{eq:to_be_in_matrix_forms} is equal to the Vandermonde-like determinant
  \begin{equation} \label{eq:det_of_f^I}
    \det \left( f^{I_l}_{j, k}(w_k) \right)^N_{j, k = 1}, \quad \text{where} \quad f^{I_l}_{j, k}(x) =
    \begin{cases}
      x^j & \text{if $k \in I_l$}, \\
      (qx)^{j - 1} & \text{if $k \in \{ 1, \dotsc, N \} \setminus I_l$}.
    \end{cases}
  \end{equation}
  Summing up all the $\binom{N}{l}$ terms on the right-hand side of \eqref{eq:to_be_in_matrix_forms} with each term expressed by \eqref{eq:det_of_f^I}, we have that the result is the $z^l$ coefficient of the polynomial in $z$
  \begin{equation} \label{eq:det_poly_in_z}
    \det \left( (zw_k + q^{j - 1})w^{j - 1}_k \right)^N_{j, k = 1}.
  \end{equation}
  Using elementary row operations, we can simplify the determinant in \eqref{eq:det_poly_in_z}, and write it as
  \begin{equation}
    \begin{split}
      \det \left( (zw_k + q^{j - 1})w^{j - 1}_k \right)^N_{j, k = 1} = {}& \sum^N_{j = 0} q^{(N - l)(N - l - 1)/2} \det \left( g^{(l)}_j(w_k) \right)^N_{j, k = 1} z^l \\
      = {}& \sum^N_{j = 0} q^{(N - l)(N - l - 1)/2} e_l(w_1, \dotsc, w_N) \Delta_N(w_1, \dotsc, w_N) z^l,
    \end{split}
  \end{equation}
  where $e_l$ and $\Delta_N$ are defined in \eqref{eq:elem_symm} and \eqref{eq:Vandermonde}, and
  \begin{equation}
    g^{(l)}_j(x) =
    \begin{cases}
      x^{j - 1} & \text{if $j \leq N - l$}, \\
      x^j & \text{otherwise}.
    \end{cases}
  \end{equation}
  Thus we conclude that
  \begin{equation}
    \Gamma(w_1, \dotsc, w_N; l) = e_l(w_1, \dotsc, w_N) B_N(w_1, \dotsc, w_N) \prod^N_{k = 1} \frac{1}{w_k}.
  \end{equation}
  Hence by \eqref{eq:express_Gamma_m(w)} and the identity $\sum^N_{l = 1} e_l(w_1, \dotsc, w_N) z^l = \prod^N_{k = 1} (1 + zw_k) - 1$,
  \begin{equation} \label{eq:ess_term_of_tilde_I}
    \Gamma_m(w_1, \dotsc, w_N) = \left[ \prod^N_{k = 1} \left( \frac{b_m - w_k}{b_m} \right) - 1 \right] B_N(w_1, \dotsc, w_N) \prod^N_{k = 1} \frac{1}{w_k}.
  \end{equation}
  Finally, substituting \eqref{eq:ess_term_of_tilde_I} into \eqref{eq:simplified_I^n_Y}, and using the telescoping trick, we find that the left-hand side of \eqref{eq-identity} is
  \begin{equation} \label{eq:final_I_Y(M;F)_telescoping}
    \begin{split}
      \\
      {}& \sum^{\infty}_{m = M + 1} \dashint_{C} \frac{dw_1}{w_1} \dotsi \dashint_{C} \frac{dw_N}{w_N} \left[ \prod^N_{k = 1} \left( \frac{b_m - w_k}{b_m} \right) - 1 \right] B_N(w_1, \dotsc, w_N) \\
      & \times \prod^N_{r = 1} \left[ \prodprime^m_{k = y_r} \left( \frac{b_k}{b_k - w_r} \right) e^{-w_r t} \right] F(w_1, \dotsc, w_N) \\
      = {}& \dashint_{C} \frac{dw_1}{w_1} \dotsi \dashint_{C} \frac{dw_N}{w_N} B_N(w_1, \dotsc, w_N) \prod^N_{r = 1} \left[ \prodprime^M_{k = y_r} \left( \frac{b_k}{b_k - w_r} \right) e^{-w_r t} \right] F(w_1, \dotsc, w_N).
    \end{split}
  \end{equation}
  This completes the proof, and the convergence of \eqref{eq:final_I_Y(M;F)_telescoping} is due to that the radius of $C$ is larger than $\sup_{k \in \intZ}(b_k)$.
\end{proof}

\paragraph{Rightmost particle}

Similarly, we can find the distribution of $x_1(t)$, the rightmost particle. The argument is parallel, and we stress the difference between the two derivations.

\begin{proposition} \label{prop:rightmost}
  For the $N$-particle $q$-TAZRP model with initial condition $Y = (y_1, \dotsc, y_N)$, the rightmost particle $x_1$ satisfies
  \begin{multline}
    \mathbb{P}_Y(x_1(t) \leq M) = (-1)^N q^{N(N - 1)/2} \dashint_{\tilde{C}_1} \frac{dw_1}{w_1} \dotsi \dashint_{\tilde{C}_N} \frac{dw_N}{w_N} B_N(w_1, \dotsc, w_N) \\
    \times \prod^N_{j = 1} \left[ \prodprime^M_{k = y_j} \left( \frac{b_k}{b_k - w_j} \right) e^{-w_j t} \right],
  \end{multline}
  where the contours $\tilde{C}_1, \dotsc, \tilde{C}_N$ are specified as in Theorem \ref{main-theorem-1}.
\end{proposition}
The proof of this proposition is by expressing, similar to \eqref{eq-marginal},
\begin{equation} \label{eq-marginal-1}
  \mathbb{P}_Y(x_1(t) \leq M) = \sum_{X = (x_1,\dotsc, x_N) \text{ with } -\tilde{M} \leq x_N \leq \cdots \leq x_1\leq M} P_Y(X;t)
\end{equation}
where $-\tilde{M}$ is a small enough integer, say, $-\tilde{M} < y_N$, and then using the following technical result with $F(w_1, \dotsc, w_N) = 1$.
\begin{proposition} \label{thm:right_most}
  Let $q \in (0,1)$, $Y = (y_1, \dotsc, y_N) \in \intZ^N$, positive numbers $b_k$ for $k \in \mathbb{Z}$ and meromorphic function $F(w_1, \dotsc, w_N)$ be defined as in Proposition \ref{thm:left_most}. Let $C$ be a positively oriented contour containing all $b_k$, and we further require that $F(w_1, \dotsc, w_N)$ is continuous if all $w_k$ are on or within $C$. Moreover, let $\tilde{C}_1, \dotsc, \tilde{C}_N$ be positively oriented contours such that each of them encloses all $b_k$, none of them encloses $0$, and $\tilde{C}_j$ enclose $q \cdot \tilde{C}_i$ if $j > i$. We also require that as the contours for $w_k$ deform gradually from $C$ to $\tilde{C}_k$, the meromorphic function $F(w_1, \dotsc, w_N)$ does not meet any singularity. Suppose $I(x_1, \dotsc, x_N)$ are defined as in \eqref{eq:defn_I(x)} of Proposition \ref{thm:left_most} with the contour $C$ specified in this proposition. We have
  \begin{multline} \label{eq:general_n_rightmost}
    \sum_{\substack{X = (x_1,\dotsc, x_N) \\ \text{with } -\tilde{M} \leq x_N \leq \cdots \leq x_1\leq M}} I(x_1, \dotsc, x_N) = (-1)^N q^{N(N - 1)/2} \dashint_{\tilde{C}_1} \frac{dw_1}{w_1} \dotsi \dashint_{\tilde{C}_N} \frac{dw_N}{w_N} \\
    \times F(w_1, \dotsc, w_N) B_N(w_1, \dotsc, w_N) \prod^N_{j = 1} \left[ \prodprime^M_{k = y_j} \left( \frac{b_k}{b_k - w_j} \right) e^{-w_j t} \right],
  \end{multline}
  where $-\tilde{M}$ is a small enough integer, say, $-\tilde{M} < \min(y_1, \dotsc, y_N)$.
\end{proposition}

\begin{proof}[Proof of Proposition \ref{thm:right_most}]
  We prove by the induction on $N$. For the $N = 1$ case, we use the telescoping trick as for \eqref{eq:n=1_left_most}, and have
  \begin{equation}
    \begin{split}
      \sum^M_{x_1 = -\tilde{M}} \left( \frac{-1}{b_{x_1}} \prodprime^{x_1}_{k = y_1} \frac{b_k}{b_k - w_1} \right) = {}& \frac{1}{w} \prodprime^{-\tilde{M} - 1}_{k = y_1} \frac{b_k}{b_k - w_1} - \frac{1}{w} \prodprime^M_{k = y_1} \frac{b_k}{b_k - w_1} \\
      = {}& \frac{1}{w} \prod^{y_1 - 1}_{k = -\tilde{M}} \frac{b_k - w_1}{b_k} - \frac{1}{w} \prodprime^M_{k = y_1} \frac{b_k}{b_k - w_1},
    \end{split}
  \end{equation}
  where we use that $-\tilde{M} < y_1$. Since by assumption, $F(w_1)$ has no pole for $w_1$ on or within the contour $C$, we have
  \begin{equation}
    \begin{split}
      \sum^M_{x_1 = -\tilde{M}} I(x_1) = {}& \dashint_C dw_1 F(w_1) \left[ \sum^M_{x_1 = -\tilde{M}_1} \left( \frac{-1}{b_{x_1}} \prodprime^{x_1}_{k = y_1} \frac{b_k}{b_k - w_1} \right) \right] e^{-w_1 t} \\
      = {}& \dashint_{\tilde{C}_1} dw_1 F(w_1) \left[ \sum^M_{x_1 = -\tilde{M}_1} \left( \frac{-1}{b_{x_1}} \prodprime^{x_1}_{k = y_1} \frac{b_k}{b_k - w_1} \right) \right] e^{-w_1 t} \\
      = {}& \dashint_{\tilde{C}_1} \frac{dw_1}{w_1} F(w_1) \left( \prod^{y_1 - 1}_{k = -\tilde{M}} \frac{b_k - w_1}{b_k} \right) e^{-w_1 t} - \dashint_{\tilde{C}_1} \frac{dw_1}{w_1} F(w_1) \left( \prodprime^M_{k = y_1} \frac{b_k}{b_k - w_1} \right) e^{-w_1 t} \\
      = {}& - \dashint_{\tilde{C}_1} \frac{dw_1}{w_1} F(w_1) \left( \prodprime^M_{k = y_1} \frac{b_k}{b_k - w_1} \right) e^{-w_1 t}.
    \end{split}
  \end{equation}
  Hence the $N = 1$ case of \eqref{eq:general_n_rightmost} is proved.

  Assuming that \eqref{eq:general_n_rightmost} holds if $N$ is replaced by  $1, 2, \dotsc, N - 1$, analogous to \eqref{eq:1135-219} we write the left-hand side of \eqref{eq:general_n_rightmost} as
  \begin{equation}
    \sum^{M - 1}_{m = -\tilde{M} - 1} \sum^{N - 1}_{l = 0} J^{(l)}(m), \quad \text{where} \quad J^{(l)}(m) = \sum_{\tilde{M} \leq x_N \leq \dotsb \leq x_{N - l + 1} \leq m} I(\underbrace{m + 1, \dotsc, m + 1}_{x_1 =\dotsb = x_{N-l} = m + 1}, x_{N - l + 1}, \dotsc, x_N).
  \end{equation}
  Then analogous to \eqref{eq:1235-219}, we have with the index $U = (u_1, \dotsc, u_l) \in \Weylchamber[l]$,
  \begin{equation}
    \begin{split}
      J^{(l)}(m) = {}& \frac{(-1/b_{m + 1})^{N - l}}{[N - l]_q!}\sum_{ -\tilde{M} \leq u_l \leq \dotsb \leq u_1 \leq m} \frac{1}{W(U)} \left( \prod^l_{k = 1} \frac{-1}{b_{u_k}} \right) \dashint_{C} dw_1 \dotsi \dashint_{C} dw_N ~ F(w_1, \dotsc, w_N) \\
      & \times  \sum_{\sigma \in S_N} A_{\sigma}(w_1, \dotsc, w_N) \prod^{N - l}_{j = 1} \left[ \prodprime^{m + 1}_{k = y_{\sigma(j)}} \left( \frac{b_k}{b_k - w_{\sigma(j)}} \right) e^{-w_{\sigma(j)} t} \right] \\
      & \times \prod^{l}_{j = 1} \left[ \prodprime^{u_j}_{k = y_{\sigma(j)}} \left( \frac{b_k}{b_k - w_{\sigma(N - l + j)}} \right) e^{-w_{\sigma(N - l + j)} t} \right],
    \end{split}
  \end{equation}
  such that $J^{(0)}(m)$ degenerates to $I(m + 1, \dotsc, m + 1)$ in \eqref{eq:defn_I(x)}. To evaluate $J^{(l)}(m)$, we also decompose $S_N$ into $\binom{N}{l}$ disjoint subsets $S_N(I_l)$ that are defined in \eqref{eq:defn_S_n(I)}. Similar to \eqref{141-220}, we have
  \begin{multline} \label{eq:J^(l)(m)_defn}
    J^{(l)}(m) = \frac{(-1/b_{m + 1})^{N - l}}{[N - l]_q!}\sum_{I_l \subseteq \{ 1, \dotsc, N \}} \sum_{\tau \in S_{N - l}} \dashint_{C} dw_{j_1} \dotsi \dashint_{C} dw_{j_{N - l}} A_{\tau}(w_{j_1}, \dotsc, w_{j_{N - l}}) \\
    \times \prod^{N - l}_{r = 1}  \prodprime^{m + 1}_{k = y_{j_{\tau(r)}}} \left( \frac{b_k}{b_k - w_{j_{\tau(r)}}} \right) e^{-w_{j_{\tau(r)}} t} \\
    \times \left(
      \begin{aligned}
        & \sum_{ -\tilde{M} \leq u_l \leq \dotsb \leq u_1 \leq m}  \frac{1}{W(U)} \left( \prod^l_{k = 1} \frac{-1}{b_{u_k}} \right) \dashint_{C} dw_{i_1} \dotsi \dashint_{C} dw_{i_l} F(w_1, \dotsc, w_N) \\
        & \times G_{I_l}(w_1, \dotsc, w_N) \sum_{\lambda \in S_l} A_{\lambda}(w_{i_1}, \dotsc, w_{i_l}) \prod^l_{r = 1}  \prodprime^{u_r}_{k = y_{j_{\lambda(r)}}} \left( \frac{b_k}{b_k - w_{i_{\lambda(r)}}} \right) e^{-w_{i_{\lambda(r)}} t}
      \end{aligned}
    \right),
  \end{multline}
  where the notations $I_l = \{ i_1, \dotsc, i_l \}$ and $\{ 1, \dotsc, N \} \setminus I_l = \{ j_1, \dotsc, j_{N - l} \}$ are the same as in \eqref{141-220}, and the decompositions of $\sigma$ into $\lambda$ and $\tau$, and of $A_{\sigma}$ into the product of $G_{I_l} A_{\lambda} A_{\tau}$ are given by \eqref{eq:defn_phi_I(lambda_tau)}, \eqref{eq:decomp_A_{sigma}} and \eqref{eq:defn_G_I_l}.

  Applying the induction hypothesis to \eqref{eq:J^(l)(m)_defn}, we have
  \begin{multline} \label{eq:J^(l)(m)_mixed}
    J^{(l)}(m) = \frac{(-1/b_{m + 1})^{N - l}}{[N - l]_q!}\sum_{I_l \subseteq \{ 1, \dotsc, N \}} \sum_{\tau \in S_{N - l}} \dashint_{C} dw_{j_1} \dotsi \dashint_{C} dw_{j_{N - l}} A_{\tau}(w_{j_1}, \dotsc, w_{j_{N - l}}) \\
    \times \prod^{N - l}_{r = 1}  \prodprime^{m + 1}_{k = y_{j_{\tau(r)}}} \left( \frac{b_k}{b_k - w_{j_{\tau(r)}}} \right) e^{-w_{j_{\tau(r)}} t} \\
    \times \left(
    \begin{aligned}
      (-1)^l q^{l(l - 1)/2} \dashint_{\tilde{C}_{i_1}} \frac{dw_{i_1}}{w_{i_1}} \dotsi \dashint_{\tilde{C}_{i_l}} \frac{dw_{i_l}}{w_{i_l}} & F(w_1, \dotsc, w_N) G_{I_l}(w_1, \dotsc, w_N) B_l(w_{i_1}, \dotsc, w_{i_l}) \\
    & \times \prod^l_{r = 1} \left[ \prodprime^m_{k = y_{i_r}} \left( \frac{b_k}{b_k - w_{i_r}} \right) e^{-w_{i_r} t} \right]
    \end{aligned}
    \right).
  \end{multline}
  By the property of the contours $C$ and $\tilde{C}_j$, we have that the integral contours for $w_{j_r}$ in \eqref{eq:J^(l)(m)_mixed} can be deformed from $C$ to $\tilde{C}_{j_r}$ for each $r = 1, \dotsc, N - l$. Then by \eqref{eq:basic_id}, we further have
  \begin{multline}
    J^{(l)}(m) = \left( \frac{-1}{b_{m + 1}} \right)^{N - l} (-1)^l q^{l(l - 1)/2} \dashint_{\tilde{C}_1} dw_1 \dotsi \dashint_{\tilde{C}_N} dw_N \Sigma_m(w_1, \dotsc, w_N; l) \\
    \times F(w_1, \dotsc, w_N) \prod^N_{r = 1} \left[ \prodprime^{m + 1}_{k = y_r} \left( \frac{b_k}{b_k - w_r} \right) e^{-w_r t} \right],
  \end{multline}
  where, analogous to $\Gamma(w_1, \dotsc, w_N; l)$ in \eqref{eq:defn_Gamma(w;l)}, we define
  \begin{multline}
    \Sigma_m(w_1, \dotsc, w_N; l) = \sum_{I_l \subseteq \{ 1, \dotsc, N \} } G_I(w_1, \dotsc, w_N)\prod^l_{k = 1} \frac{(b_{m + 1} - w_{i_k})}{b_{m + 1} w_{i_k}} \\
    \times B_l(w_{i_1}, w_{i_2} \dotsc, w_{i_l}) B_{N - l}(w_{j_1}, \dotsc, w_{j_{N - l}}).
  \end{multline}

  Analogous to $\Gamma_m(w_1, \dotsc, w_N)$ in \eqref{eq:express_Gamma_m(w)}, we define
  \begin{equation} \label{eq:new_defn_Sigma_m}
    \Sigma_m(w_1, \dotsc, w_N) = \sum^N_{l = 0} \left( \frac{-1}{b_{m + 1}} \right)^{N - l} (-1)^l q^{l(l - 1)/2} \Sigma_m(w_1, \dotsc, w_N; l),
  \end{equation}
  and then similar to \eqref{eq:simplified_I^n_Y}, we have that the left-hand side of \eqref{eq:general_n_rightmost} is expressed as
  \begin{multline} \label{eq:simplified_J^n_Y}
    \sum^{M - 1}_{m = -\tilde{M} - 1} \dashint_{\tilde{C}} dw_1 \dotsi \dashint_{\tilde{C}_N} dw_N \left( \Sigma_m(w_1, \dotsc, w_N) - (-1)^N q^{N(N - 1)/2} \prod^N_{k = 1} \frac{b_{m + 1} - w_k}{b_{m + 1} w_k} B_N(w_1, \dotsc, w_N) \right) \\
    \times \prod^N_{r = 1} \left[ \prodprime^{m + 1}_{k = y_r} \left( \frac{b_k}{b_k - w_r} \right) e^{-w_r t} \right]  F(w_1, \dotsc, w_N).
  \end{multline}
  (Formally, \eqref{eq:simplified_J^n_Y} has an extra term comparing with \eqref{eq:simplified_I^n_Y}, but this difference results from that $\Sigma_m(w_1, \dotsc, w_N)$ in \eqref{eq:new_defn_Sigma_m} is the sum of $N + 1$ terms, while $\Gamma_m(w_1, \dotsc, w_N)$ in \eqref{eq:express_Gamma_m(w)} is the sum of $N$ terms.) Similar to \eqref{eq:to_be_in_matrix_forms}, we have
  \begin{multline} \label{eq:to_be_in_matrix_forms_rightmost}
    q^{(N - l)(N - l - 1)/2} (-1)^{N(N - 1)/2} \prod^N_{k = 1} w_k \prod_{1 \leq i < j \leq N}(qw_i - w_j) \Sigma_m(w_1, \dotsc, w_N; l) \\
    = \sum_{I_l \subseteq \{ 1, \dotsc, N \}} (-1)^{(i_1 + \dotsb + i_l) - l(l + 1)/2} \prod^l_{k = 1} (1 - w_{i_k}/b_{m + 1}) \prod^{N - l}_{k = 1} w_{j_k} \\
    \times \prod_{1 \leq k < r \leq l} (w_{i_r} - w_{i_k}) \prod_{1 \leq k < r \leq N - l} (qw_{j_r} - qw_{j_k}) \prod^l_{k = 1} \prod^{N - l}_{r = 1} (qw_{j_r} - w_{i_k}).
  \end{multline}
  It is not hard to check that for any intex set $I_l$, the term in the sum \eqref{eq:to_be_in_matrix_forms_rightmost} is equal to
  \begin{equation} \label{eq:determinant_Sigma(m;l)}
    \det \left( f^{I_l}_{j, k}(w_k) \right)^N_{j, k = 1}, \quad \text{where} \quad f^{I_l}_{j, k}(x) =
    \begin{cases}
      (1 - \frac{x}{b_{m + 1}}) x^j & \text{if $k \in I_l$}, \\
      x (qx)^{j - 1} & \text{if $k \in \{ 1, \dotsc, N \} \setminus I_l$}.
    \end{cases}
  \end{equation}
  Then by summing up all the $\binom{N}{l}$ terms, we have that the left-hand side of \eqref{eq:to_be_in_matrix_forms_rightmost} is the $z^l$ coefficient of the following polynomial in $z$
  \begin{equation} \label{eq:det_poly_in_z_rightmost}
    \det \left( [z(1 - w_k/b_{m + 1}) + q^{j - 1}w_k] w^{j - 1}_k \right)^N_{j, k = 1}.
  \end{equation}
Hence
  \begin{equation} \label{eq:Sigma(m)}
    \begin{split}
      & (-q)^{N(N - 1)/2} \prod^N_{k = 1} w_k \prod_{1 \leq i < j \leq N}(qw_i - w_j) \Sigma_m(w_1, \dotsc, w_N) \\
      = {}& \sum^N_{l = 0} \left( \frac{-1}{b_{m + 1}} \right)^N b^l_{m + 1} q^{(N - 1)l} \left( \text{$z^l$ coefficient of } \det \left( [z(1 - w_k/b_{m + 1}) + q^{j - 1}w_k] w^{j - 1}_k \right)^N_{j, k = 1} \right) \\
      = {}& \left( \frac{-1}{b_{m + 1}} \right)^n \left. \det \left( [z(1 - w_k/b_{m + 1}) + q^{j - 1}w_k] w^{j - 1}_k \right)^N_{j, k = 1} \right\rvert_{z = -q^{N - 1} b_{m + 1}} \\
      = {}& (-1)^N q^{N(N - 1)} \Delta_N(w_1, \dotsc, w_N).
    \end{split}
  \end{equation}
  Then by \eqref{eq:Sigma(m)} we can express $\Sigma_m(w_1, \dotsc, w_N)$ in terms of $B_N(w_1, \dotsc, w_N)$, and by plugging this expression into \eqref{eq:simplified_J^n_Y} we have that the left-hand side of \eqref{eq:general_n_rightmost} is
  \begin{multline}
    (-1)^N q^{N(N - 1)/2} \sum^{M - 1}_{m = -\tilde{M} - 1} \dashint_{\tilde{C}_1} \frac{dw_1}{w_1} \dotsi \dashint_{\tilde{C}_N} \frac{dw_N}{w_N} \left[ 1 - \left( \prod^N_{k = 1} \frac{b_{m + 1} - w_k}{b_{m + 1}} \right) \right] B_N(w_1, \dotsc, w_N) \\
    \times \prod^N_{r = 1} \left[ \prodprime^{m + 1}_{k = y_r} \left( \frac{b_k}{b_k - w_r} \right) e^{-w_r t} \right]  F(w_1, \dotsc, w_N) \\
    = (-1)^N q^{N(N - 1)/2} \left( \dashint_{\tilde{C}_1} \frac{dw_1}{w_1} \dotsi \dashint_{\tilde{C}_N} \frac{dw_N}{w_N} \prod^N_{r = 1} \left[ \prodprime^M_{k = y_r} \left( \frac{b_k}{b_k - w_r} \right) e^{-w_r t} \right]  F(w_1, \dotsc, w_N) \right. \\
    - \left. \dashint_{\tilde{C}_1} \frac{dw_1}{w_1} \dotsi \dashint_{\tilde{C}_N} \frac{dw_N}{w_N} \prod^N_{r = 1} \left[ \prodprime^{-\tilde{M} - 1}_{k = y_r} \left( \frac{b_k}{b_k - w_r} \right) e^{-w_r t} \right]  F(w_1, \dotsc, w_N) \right),
  \end{multline}
  where we use the telescoping trick for the identity. Hence we prove Theorem \ref{thm:right_most} by noting that
  \begin{equation}
    \dashint_{\tilde{C}_1} \frac{dw_1}{w_1} \dotsi \dashint_{\tilde{C}_N} \frac{dw_N}{w_N} \prod^N_{r = 1} \left[ \prodprime^{-\tilde{M} - 1}_{k = y_r} \left( \frac{b_k}{b_k - w_r} \right) e^{-w_r t} \right]  F(w_1, \dotsc, w_N) = 0,
  \end{equation}
  which is due to that $\tilde{M} < \min(y_1, \dotsc, y_N)$ and $F(w_1, \dotsc, w_N)$ has no singularity if all $w_k$ are on or within $\tilde{C}_k$.
\end{proof}

\subsubsection{Relation between the two types of nested contour integrals.}

Using the residue formula inductively, we have that if the contours $C$ and $\tilde{C}_1, \dotsc, \tilde{C}_N$ are defined as in Theorem \ref{main-theorem-1}, then
\begin{multline} \label{eq:identical_in_nested}
  \dashint_C \frac{dw_1}{w_1} \dotsi \dashint_C \frac{dw_N}{w_N} B_N(w_1, \dotsc, w_N) \prod^N_{j = 1} \left[ \prodprime^M_{k = y_j} \left( \frac{b_k}{b_k - w_j} \right) e^{-w_j t} \right] = \\
  \sum^N_{l = 0} \sum_{I_l \subseteq \{ 1, \dotsc, N \}} q^{(i_1 + \dotsb + i_l) - (2N -l + 1)l/2} \\
  \times \dashint_{\tilde{C}_{i_1}} \frac{dw_{i_1}}{w_{i_1}} \dotsi \dashint_{\tilde{C}_{i_l}} \frac{dw_{i_l}}{w_{i_l}} B_l(w_{i_1}, \dotsc, w_{i_l}) \prod^l_{r = 1} \left[ \prodprime^M_{k = y_{i_r}} \left( \frac{b_k}{b_k - w_{i_r}} \right) e^{-w_{i_r} t} \right],
\end{multline}
where $I_l = \{ i_1, \dotsc, i_l \}$ with $i_k$ in ascending order, and the term corresponding to $l = 0$ and $I_l = \emptyset$ is taken as $1$. On the other hand, we have
\begin{multline} \label{eq:nested_in_identical}
  \dashint_{\tilde{C}_1} \frac{dw_1}{w_1} \dotsi \dashint_{\tilde{C}_N} \frac{dw_N}{w_N} B_N(w_1, \dotsc, w_N) \prod^N_{j = 1} \left[ \prodprime^M_{k = y_j} \left( \frac{b_k}{b_k - w_j} \right) e^{-w_j t} \right] = \\
  \sum^N_{l = 0} \sum_{I_l \subseteq \{ 1, \dotsc, N \}} (-1)^{N - l} q^{(i_1 + \dotsb + i_l) - l - N(N - 1)/2} \\
  \times \dashint_C \frac{dw_{i_1}}{w_{i_1}} \dotsi \dashint_C \frac{dw_{i_l}}{w_{i_l}} B_l(w_{i_1}, \dotsc, w_{i_l}) \prod^l_{r = 1} \left[ \prodprime^M_{k = y_{i_r}} \left( \frac{b_k}{b_k - w_{i_r}} \right) e^{-w_{i_r} t} \right],
\end{multline}
where the term corresponding to $l = 0$ and $I_l = \emptyset$ is taken as $q^{-N(N - 1)/2}$.

We remark that relation \eqref{eq:identical_in_nested} is analogous to the expression of $\tilde{\mu}_k$ as a linear combination of $\mu_j$ in \cite[Proposition 3.12]{Borodin-Corwin13}.

\subsubsection{Proof of Theorem \ref{main-theorem-1}}

In this subsection we compute the distribution function $\Prob_Y(x_n(t) \leq M)$ and $\Prob_Y(x_{N - n + 1}(t) > M)$, where $n = 1, 2, \dotsc, N$. Let us write
\begin{equation} \label{eq:general_position_decomp}
  \Prob_Y(x_n(t) \leq M) = \sum^{n - 1}_{l = 0} P^{(l)}_Y(M), \quad \Prob_Y(x_{N - n + 1}(t) > M) = \sum^{N}_{l = N - n + 1} P^{(l)}_Y(M),
\end{equation}
where
\begin{equation}
  P^{(0)}_Y(M) = \Prob_Y(x_1(t) \leq M), \quad P^{(N)}_Y(M) = \Prob_Y(x_N > M),
\end{equation}
and for $l = 1, \dotsc, N - 1$,
\begin{equation}
  P^{(l)}_Y(M) = \mathbb{P}_Y(x_{l + 1}(t) \leq M \text{ and } x_l(t) > M).
\end{equation}
Recall the notations $I_l = \{ i_1, \dotsc, i_l \} \subseteq \{ 1, 2, \dotsc, N \}$ with $i_1 < \dotsb < i_l$, the complement set $I^c_l = \{1, \dotsc, N \} \setminus I_l = \{ j_1, \dotsc, j_{N - l} \}$ with $j_1 < \dotsb < j_{N - l}$, the subset $S_N(I_l)$ defined in \eqref{eq:defn_S_n(I)} of the permutation group $S_N$, the mapping $\varphi_{I_l}(\cdot \mid \cdot)$ defined in \eqref{eq:defn_phi_I(lambda_tau)} that maps $S_l \times S_{N - l}$ bijectively to $S_N(I_l)$. We decompose
\begin{equation} \label{eq:decomp_P^l_Y(M)}
  P^{(N - l)}_Y(M) = \sum_{I_l \subseteq \{ 1, \dotsc, N \}} P^{(N - l)}_Y(M; I_l),
\end{equation}
where, with indices $U = (u_1, \dotsc, u_{N - l}) \in \Weylchamber[N - l]$ and $V = (v_1, \dotsc, v_l) \in \Weylchamber[l]$,
\begin{multline} \label{eq:first_expression_P^(l)_Y(M;l)}
  P^{(N - l)}_Y(M; I_l) = \sum_{-\tilde{M} \leq v_l \leq \dotsb \leq v_1 \leq M < u_{N - l} \leq \dotsb \leq u_1} \frac{1}{W(U)} \prod^{N - l}_{j = 1} \left( \frac{-1}{b_{u_j}} \right) \frac{1}{W(V)} \prod^l_{j = 1} \left( \frac{-1}{b_{v_j}} \right) \\
  \times \dashint_C dw_1 \dotsi \dashint_C dw_N \sum_{\lambda \in S_l, \tau \in S_{N - l}} A_{\varphi_{I_l}(\lambda \mid \tau)}(w_1, \dotsc, w_N) \\
  \times \prod^{N - l}_{r = 1} \left[ \prod^{u_r}_{k = y_{j_{\tau(r)}}} \left( \frac{b_k}{b_k - w_{j_{\tau(r)}}} \right) e^{-w_{j_{\tau(r)}} t} \right] \prod^l_{r = 1} \left[ \prod^{v_r}_{k = y_{i_{\lambda(r)}}} \left( \frac{b_k}{b_k - w_{i_{\lambda(r)}}} \right) e^{-w_{i_{\lambda(r)}} t} \right],
\end{multline}
and $P^{(N - l)}_Y(M; I_l)$ degenerates into \eqref{eq-marginal-1} if $l = N$ and $I_l = \{ 1, \dotsc, N \}$, or degenerates into \eqref{eq-marginal-1} if $l = 0$ and $I_l = \emptyset$. In \eqref{eq:first_expression_P^(l)_Y(M;l)} we assume the contour $C$ and constant $-\tilde{M}$ are specified in Proposition \ref{thm:right_most}. Recall that $A_{\varphi_{I_l}(\lambda \mid \tau)}(w_1, \dotsc, w_N)$ can be decomposed by \eqref{eq:decomp_A_{sigma}}. Now we apply Proposition \ref{thm:left_most} to the right-hand side of \eqref{eq:first_expression_P^(l)_Y(M;l)}, with $X$ replaced by $U$, the $N$ variables $w_1, \dotsc, w_N$ replaced by the $(N - l)$ variables $w_{j_1}, \dotsc, w_{j_{N - l}}$, and the function $F(w_{j_1}, \dotsc, w_{j_{N - l}})$ specialized by
\begin{multline}
  F(w_{j_1}, \dotsc, w_{j_{n - l}}) = \sum_{-\tilde{M} \leq v_l \leq \dotsb \leq v_1 \leq M} \frac{1}{W(V)} \prod^l_{j = 1} \left( \frac{-1}{b_{v_j}} \right) \dashint_C dw_{i_1} \dotsi \dashint_C dw_{i_l} \\
  \times G_{I_l}(w_1, \dotsc, w_N) A_{\lambda}(w_{i_1}, \dotsc, w_{i_l}) \prod^l_{r = 1} \left[ \prod^{v_r}_{k = y_{i_{\lambda(r)}}} \left( \frac{b_k}{b_k - w_{i_{\lambda(r)}}} \right) e^{-w_{i_{\lambda(r)}} t} \right],
\end{multline}
where $G_{I_l}(w_1, \dotsc, w_N)$ is defined in \eqref{eq:defn_G_I_l}. We then rewrite \eqref{eq:first_expression_P^(l)_Y(M;l)} as
\begin{multline} \label{eq:intermediate_expression_P^(l)_Y(M;l)}
  P^{(N - l)}_Y(M; I_l) = \dashint_C \frac{dw_{j_1}}{w_{j_1}} \dotsi \dashint_C \frac{dw_{j_{N - l}}}{w_{j_{N - l}}} F(w_{j_1}, \dotsc, w_{j_{N - l}}) \\
  \times B_{N - l}(w_{j_1}, \dotsc, w_{j_{N - l}}) \prod^{N - l}_{r = 1} \left[ \prod^M_{k = y_{j_r}} \left( \frac{b_k}{b_k - w_{j_r}} \right) e^{-w_{j_r} t} \right].
\end{multline}
Now we note that $F(w_{j_1}, \dotsc, w_{j_{N - l}})$ can be simplified by Proposition \ref{thm:right_most}, with $X$ replaced by $V$, the $N$ variables $w_1, \dotsc, w_N$ replaced by the $l$ variables $w_{i_1}, \dotsc, w_{i_l}$, and the function $F(w_{i_1}, \dotsc, w_{i_l})$ specialized by $G_{I_l}(w_1, \dotsc, w_N)$ in which $w_{i_1}, \dotsc, w_{i_l}$ are variables and $w_{j_1}, \dotsc, w_{j_{N - l}}$ are parameters. Thus \eqref{eq:intermediate_expression_P^(l)_Y(M;l)} is further simplified into
\begin{multline} \label{eq:hetero_integral_P^(N-l)_Y}
  P^{(N - l)}_Y(M; I_l) = (-1)^l q^{l(l - 1)/2} \dashint_C \frac{dw_{j_1}}{w_{j_1}} \dotsi \dashint_C \frac{dw_{j_{N - l}}}{w_{j_{N - l}}} \dashint_{\tilde{C}_{i_1}} \frac{dw_{i_1}}{w_{i_1}} \dotsi \dashint_{\tilde{C}_{i_l}} \frac{dw_{i_l}}{w_{i_l}} \\
  \times G_{I_l}(w_1, \dotsc, w_N) B_l(w_{i_1}, \dotsc, w_{i_l}) B_{N - l}(w_{j_1}, \dotsc, w_{j_{N - l}}) \prod^N_{j = 1} \left[ \prod^M_{k = y_j} \left( \frac{b_k}{b_k - w_j} \right) e^{-w_j t} \right].
\end{multline}
Up to here $P^{(N - l)}_Y(M; I_l)$ is expressed by a single, but heterotypic contour integral. We can use the residue theorem to express it into a sum of multiple integrals over $C$, or a sum of multiple integrals over $\tilde{C}_1, \dotsc, \tilde{C}_N$. The former approach leads to the proof of \eqref{eq:distr_X_m}, and the latter approach leads to the proof of \eqref{eq:distr_X_N-m+1}.

\paragraph{The $n$-th rightmost particle}

We first prove \eqref{eq:distr_X_N-m+1}. For any pair of disjoint subsets $I' = \{ i'_1, \dotsc, i'_l \}$ and $J' = \{ j'_1, \dotsc, j'_m \}$ of $\{ 1, \dotsc, N \}$ with $i'_1 < \dotsb < i'_l$ and $j'_1 < \dotsb < j'_m$, we denote
\begin{equation}
  C_{I', J'} = \prod_{i < i' \in I'} \frac{w_i - w_{i'}}{qw_i - w_{i'}} \prod_{j < j' \in J'} \frac{w_j - w_{j'}}{qw_j - w_{j'}} \prod_{i \in I', j \in J', \text{ and } i < j} -\frac{qw_j - w_i}{qw_i - w_j},
\end{equation}
so that if we take $I' = I_l = \{ i_1, \dotsc, i_l \}$ and $J = I^c_l = \{ j_1, \dotsc, j_{N - l} \}$ as above \eqref{eq:defn_phi_I(lambda_tau)}, then
\begin{equation*}
C_{I_l, I^c_l} = G_{I_l}(w_1, \dotsc, w_N) B_l(w_{i_1}, \dotsc, w_{i_l}) B_{N - l}(w_{j_1}, \dotsc, w_{j_{N - l}}).
\end{equation*} Using the residue theorem to \eqref{eq:hetero_integral_P^(N-l)_Y} and deform the contours for $w_{j_r}$ ($r = 1, \dotsc, N - l$) from $C$ to $\tilde{C}_{j_r}$ one by one, we have ($\lvert I' \rvert$ means the number of elements in set $I'$)
\begin{multline} \label{eq:identical_in_nested_general_l}
  P^{(N - l)}_Y(M; I_l) = (-1)^l q^{l(l - 1)/2} \sum^N_{r = l} \sum_{\substack{I' = \{ i'_1, \dotsc, i'_r \} \subseteq \{ 1, \dotsc, N \} \\ \lvert I' \rvert = r \text{ and } I' \supseteq I_l}} q^{(i'_1 + \dotsb + i'_r) - (2N - r + 1)r/2} \\
  \times \dashint_{\tilde{C}_{i'_1}} \frac{dw_{i'_1}}{w_{i'_1}} \dotsi \dashint_{\tilde{C}_{i'_r}} \frac{dw_{i'_r}}{w_{i'_r}} C_{I_l, I' \setminus I_l} \prod^r_{j = 1} \left[ \prodprime^M_{k = y_{i'_j}} \left( \frac{b_k}{b_k - w_{i'_j}} \right) e^{-w_{i'_j} t} \right].
\end{multline}
First we use \eqref{eq:identical_in_nested_general_l} and \eqref{eq:decomp_P^l_Y(M)} to derive
\begin{multline} \label{eq:sum_identical_in_nested_general_l}
  P^{(N - l)}_Y(M) = (-1)^l q^{l(l - 1)/2} \sum^N_{r = l} \sum_{\substack{I' = \{ i'_1, \dotsc, i'_r \} \\ I' \in \{ 1, \dotsc, N \} \text{ and } \lvert I' \rvert = r}} q^{(i'_1 + \dotsb + i'_r) - (2N - r + 1)r/2} \\
  \times \dashint_{\tilde{C}_{i'_1}} \frac{dw_{i'_1}}{w_{i'_1}} \dotsi \dashint_{\tilde{C}_{i'_r}} \frac{dw_{i'_r}}{w_{i'_r}} \left( \sum_{I'' \subseteq I' \text{ and } \lvert I'' \rvert = l} C_{I'', I' \setminus I''} \right) \prod^r_{j = 1} \left[ \prodprime^M_{k = y_{i'_j}} \left( \frac{b_k}{b_k - w_{i'_j}} \right) e^{-w_{i'_j} t} \right].
\end{multline}

Below we compute the sum $\sum_{I'' \subseteq I' \text{ and } \lvert I'' \rvert = l} C_{I'', I' \setminus I''}$. For notational simplicity, we denote $u_j = w_{i'_j}$ for $j = 1, \dotsc, r$. Then we have that for any $I'' \subseteq I'$, analogous to \eqref{eq:to_be_in_matrix_forms},
\begin{equation} \label{eq:C_{III}_by_det(h)}
  (-1)^{r(r - 1)/2} C_{I'', I' \setminus I''} = \frac{q^{-l(l - 1)/2}}{\prod_{1 \leq i < j \leq r} (qu_i - u_j)} \det \left( h^{I''}_{j, k}(u_k) \right)^r_{j, k = 1},
\end{equation}
where
\begin{equation} \label{eq:defn_h^I''_jk}
  h^{I''}_{j, k}(x) =
  \begin{cases}
    (qx)^{j - 1} & \text{if $i'_j \in I''$}, \\
    x^{j - 1} & \text{otherwise}
  \end{cases}
\end{equation}
which is analogous to \eqref{eq:det_of_f^I}. Then we have
\begin{equation} \label{eq:t^l_coeff_useful}
  \begin{split}
    \sum^r_{l = 0} \left[ \sum_{I'' \subseteq I' \text{ and } \lvert I'' \rvert = l} \det \left( h^{I''}_{j, k}(u_k) \right)^r_{j, k = 1} \right] t^l = {}& \det \left( (1 + q^{j - 1} t) u^{j - 1}_k \right)^r_{j, k = 1} \\
    = {}& \det \left( u^{j - 1}_k \right)^r_{j, k = 1} \prod^r_{j = 1} (1 + q^{j - 1} t) \\
    = {}& \prod_{1 \leq i < j \leq r} (u_i - u_j) \sum^r_{l = 0} q^{l(l - 1)/2} \qbinom{r}{l} t^l,
  \end{split}
\end{equation}
where the last identity is by \cite[Corollary 10.2.2(c)]{Andrews-Askey-Roy99}. Hence by using \eqref{eq:C_{III}_by_det(h)} and taking the $t^l$ coefficient in \eqref{eq:t^l_coeff_useful}, we have
\begin{equation} \label{eq:sum_C_IJ_nested}
  \begin{split}
    \sum_{I'' \subseteq I' \text{ and } \lvert I'' \rvert = l} C_{I'', I' \setminus I''} = {}& \frac{q^{-l(l - 1)/2}}{\prod_{1 \leq i < j \leq r} (qu_i - u_j)} \sum_{I'' \subseteq I' \text{ and } \lvert I'' \rvert = l} (-1)^{r(r - 1)/2} \det \left( h^{I''}_{j, k}(u_k) \right)^r_{j, k = 1} \\
    = {}& \frac{q^{-l(l - 1)/2}}{\prod_{1 \leq i < j \leq r} (qu_i - u_j)} q^{l(l - 1)/2} \qbinom{r}{l} \prod_{1 \leq i < j \leq r} (u_i - u_j) \\
    = {}& \qbinom{r}{l} B_r(u_1, \dotsc, u_r).
  \end{split}
\end{equation}

Thus we have by \eqref{eq:decomp_P^l_Y(M)}, \eqref{eq:sum_identical_in_nested_general_l} and \eqref{eq:sum_C_IJ_nested} that
\begin{multline}
  P^{(N - l)}_Y(M) = (-1)^l q^{l(l - 1)/2} \sum^N_{r = l} \sum_{\substack{I' = \{ i'_1, \dotsc, i'_r \} \\ I' \subseteq \{ 1, \dotsc, N \} \text{ and } \lvert I' \rvert = r}} q^{(i'_1 + \dotsb + i'_r) - (2N - r + 1)r/2} \qbinom{r}{l} \\
  \times \dashint_{\tilde{C}_{i'_1}} \frac{dw_{i'_1}}{w_{i'_1}} \dotsi \dashint_{\tilde{C}_{i'_r}} \frac{dw_{i'_r}}{w_{i'_r}} B_r(w_{i'_1}, \dotsc, w_{i'_r}) \prod^r_{j = 1} \left[ \prodprime^M_{k = y_{i'_j}} \left( \frac{b_k}{b_k - w_{i'_j}} \right) e^{-w_{i'_j} t} \right]
\end{multline}
for $l = 0, 1, \dotsc, N$. Hence by \eqref{eq:general_position_decomp}, we have
\begin{multline} \label{eq:almost_final_dist_form_x_m}
  P_Y(x_n(t) \leq M) = \sum^N_{r = N - n + 1} \sum_{\substack{I' = \{ i'_1, \dotsc, i'_r \} \\ I \subseteq \{ 1, \dotsc, N \} \text{ and } \lvert I' \rvert = r}} q^{(i'_1 + \dotsb + i'_r) - (2N - r + 1)r/2} \sum^r_{l = N - n + 1} (-1)^l q^{l(l - 1)/2} \qbinom{r}{l} \\
  \times \dashint_{\tilde{C}_{i'_1}} \frac{dw_{i'_1}}{w_{i'_1}} \dotsi \dashint_{\tilde{C}_{i'_r}} \frac{dw_{i'_r}}{w_{i'_r}} B_r(w_{i'_1}, \dotsc, w_{i'_r}) \prod^r_{j = 1} \left[ \prodprime^M_{k = y_{i_j}} \left( \frac{b_k}{b_k - w_{i'_j}} \right) e^{-w_{i'_j} t} \right].
\end{multline}
Furthermore, using \cite[Corollary 10.2.2(c)]{Andrews-Askey-Roy99} and then \cite[Formula (10.0.3)]{Andrews-Askey-Roy99}, we have
\begin{equation} \label{eq:first_q_binom_sum_form}
  \begin{split}
    \sum^r_{l = N - n + 1} (-1)^l q^{l(l - 1)/2} \qbinom{r}{l} = {}& -\sum^{N - n}_{l = 0} (-1)^l q^{l(l - 1)/2} \qbinom{r}{l} \\
    = {}& (-1)^{N - n + 1} q^{(N - n)(N - n + 1)/2} \qbinom{r - 1}{N - n}.
  \end{split}
\end{equation}
Hence we have
\begin{multline} \label{eq:distr_x_m(t)_general}
  \Prob_Y(x_n(t) \leq M) = (-1)^{N - n + 1} q^{(N - n)(N - n + 1)/2} \sum^N_{r = N - n + 1} \qbinom{r - 1}{N - n} \\
  \times \sum_{\substack{I' = \{ i'_1, \dotsc, i'_r \} \\ I \subseteq \{ 1, \dotsc, N \} \text{ and } \lvert I' \rvert = r}} q^{(i'_1 + \dotsb + i'_r) - (2N - r + 1)r/2} \\
  \times \dashint_{\tilde{C}_{i'_1}} \frac{dw_{i'_1}}{w_{i'_1}} \dotsi \dashint_{\tilde{C}_{i'_r}} \frac{dw_{i'_r}}{w_{i'_r}} B_r(w_{i'_1}, \dotsc, w_{i'_r}) \prod^r_{j = 1} \left[ \prodprime^M_{k = y_{i_j}} \left( \frac{b_k}{b_k - w_{i'_j}} \right) e^{-w_{i'_j} t} \right],
\end{multline}
and prove \eqref{eq:distr_X_N-m+1}.

\paragraph{The $n$-th leftmost particle}

Next we prove \eqref{eq:distr_X_m}. Similar to \eqref{eq:identical_in_nested_general_l}, using the residue theorem to \eqref{eq:hetero_integral_P^(N-l)_Y} and deform the contours for $w_{i_r}$ ($r = 1, \dotsc, l$) from $\tilde{C}_{i_r}$ to $C$ one by one, we have
\begin{multline} \label{eq:nested_in_identical_general_l}
  P^{(N - l)}_Y(M; I_l) = (-1)^l q^{l(l - 1)/2} \sum^N_{r = N - l} \sum_{\substack{J' = \{ j'_1, \dotsc, j'_r \} \subseteq \{ 1, \dotsc, N \} \\ \lvert J' \rvert = r \text{ and } J' \supseteq I^c_l}} (-1)^{N - r} q^{(j'_1 + \dotsb + j'_r) - r - N(N - 1)/2 + (N - l)(N - r)} \\
  \times \dashint_C \frac{dw_{j'_1}}{w_{j'_1}} \dotsi \dashint_C \frac{dw_{j'_r}}{w_{j'_r}} C_{J' \setminus I^c_l, I^c_l} \prod^r_{i = 1} \left[ \prodprime^M_{k = y_{j'_i}} \left( \frac{b_k}{b_k - w_{j'_i}} \right) e^{-w_{j'_i} t} \right].
\end{multline}
Using \eqref{eq:nested_in_identical_general_l} we can derive, analogous to \eqref{eq:sum_identical_in_nested_general_l},
\begin{multline} \label{eq:sum_identical_in_identical_general_l}
  P^{(N - l)}_Y(M) = (-1)^l q^{l(l - 1)/2} \sum^N_{r = N - l} \sum_{\substack{J' = \{ j'_1, \dotsc, j'_r \} \\ J' \subseteq \{ 1, \dotsc, N \} \text{ and } \lvert J' \rvert = r}} (-1)^{N - r} q^{(j'_1 + \dotsb + j'_r) - r - N(N - 1)/2 + (N - l)(N - r)} \\
  \times \dashint_{C_{j'_1}} \frac{dw_{j'_1}}{w_{j'_1}} \dotsi \dashint_{C_{j'_r}} \frac{dw_{j'_r}}{w_{j'_r}} \left( \sum_{J'' \subseteq J' \text{ and } \lvert J'' \rvert = N - l} C_{J' \setminus J'', J''} \right) \prod^r_{i = 1} \left[ \prodprime^M_{k = y_{j'_i}} \left( \frac{b_k}{b_k - w_{j'_i}} \right) e^{-w_{j'_i} t} \right].
\end{multline}
Similar to \eqref{eq:C_{III}_by_det(h)}, \eqref{eq:defn_h^I''_jk} and \eqref{eq:sum_C_IJ_nested}, we have that if we denote $v_i = w_{j'_i}$ for $i = 1, \dotsc, r$, then
\begin{equation} \label{eq:sum_C_IJ_identical}
  \begin{split}
    & \sum_{J'' \subseteq J' \text{ and } \lvert J'' \rvert = N - l} C_{J' \setminus J'', J''} \\
    = {}& \frac{q^{-(r - N + l)(r - N + r - 1)/2}}{\prod_{1 \leq i < j \leq r} (qv_i - v_j)} (-1)^{r(r - 1)/2} \det \left( h^{J' \setminus J''}_{j, k}(v_k) \right)^r_{j, k = 1} \\
    = {}& \frac{q^{-(r - N + l)(r - N + r - 1)/2}}{\prod_{1 \leq i < j \leq r} (qv_i - v_j)} q^{(r - N + l)(r - N + l - 1)/2} \qbinom{r}{r - N + l} \prod_{1 \leq i < j \leq r} (v_i - v_j) \\
    = {}& \qbinom{r}{N - l} B_r(v_1, \dotsc, v_r),
  \end{split}
\end{equation}
where
\begin{equation}
  h^{J' \setminus J''}_{j, k}(x) =
  \begin{cases}
    (qx)^{j - 1} & \text{if $j'_j \in J' \setminus J''$}, \\
    x^{j - 1} & \text{otherwise}.
  \end{cases}
\end{equation}

Thus we have by \eqref{eq:decomp_P^l_Y(M)}, \eqref{eq:sum_identical_in_identical_general_l} and \eqref{eq:sum_C_IJ_identical} that
\begin{multline}
  P^{(N - l)}_Y(M) = (-1)^l q^{l(l - 1)/2} \sum^N_{r = N - l} \sum_{\substack{J' = \{ j'_1, \dotsc, j'_r \} \\ J' \subseteq \{ 1, \dotsc, N \} \text{ and } \lvert J' \rvert = r}} (-1)^{N - r} q^{(j'_1 + \dotsb + j'_r) - r - N(N - 1)/2 + (N - l)(N - r)} \\
  \times \qbinom{r}{N - l} \dashint_C \frac{dw_{j'_1}}{w_{j'_1}} \dotsi \dashint_C \frac{dw_{j'_r}}{w_{j'_r}} B_r(w_{j'_1}, \dotsc, w_{j'_r}) \prod^r_{i = 1} \left[ \prodprime^M_{k = y_{j'_i}} \left( \frac{b_k}{b_k - w_{j'_i}} \right) e^{-w_{j'_i} t} \right]
\end{multline}
for $l = 0, 1, \dotsc, N$. Hence similar to \eqref{eq:almost_final_dist_form_x_m}, we have by \eqref{eq:general_position_decomp} that
\begin{multline} \label{eq:identical_contour_almost_summation}
  P_Y(x_{N - n + 1}(t) > M) = \sum^N_{r = N - n + 1} \sum_{\substack{J' = \{ j'_1, \dotsc, j'_r \} \\ J' \subseteq \{ 1, \dotsc, N \} \text{ and } \lvert J' \rvert = r}} (-1)^{N - r} q^{(j'_1 + \dotsb + j'_r) - r - N(N - 1)/2} \\
  \times \sum^r_{l = N - n + 1} (-1)^{N - l} q^{(N - l)(N - l - 1)/2 + l(N - r)} \qbinom{r}{l} \\
  \times \dashint_C \frac{dw_{j'_1}}{w_{j'_1}} \dotsi \dashint_C \frac{dw_{j'_r}}{w_{j'_r}} B_r(w_{j'_1}, \dotsc, w_{j'_r}) \prod^r_{i = 1} \left[ \prodprime^M_{k = y_{j_i}} \left( \frac{b_k}{b_k - w_{j'_i}} \right) e^{-w_{j'_i} t} \right].
\end{multline}
Using \eqref{eq:first_q_binom_sum_form} and the relation $\qbinom{r}{l} = \qbinom{r}{r - l}$, we have
\begin{equation}
  \begin{split}
    & \sum^r_{l = N - n + 1} (-1)^{N - l} q^{(N - l)(N - l - 1)/2 + l(N - r)} \qbinom{r}{l} \\
    = {}& (-1)^{N - r} q^{(N + r - 1)(N - r)/2} \sum^{r - (N - n + 1)}_{l = 0} (-1)^l q^{l(l - 1)/2} \qbinom{r}{l} \\
    = {}& (-1)^{N - r} q^{(N + r - 1)(N - r)/2} (-1)^{r - (N - n + 1)} q^{[r - (N - n + 1)][r - (N - n)]/2} \qbinom{r - 1}{r - (N - n + 1)} \\
    = {}& (-1)^{n - 1} q^{(N - r)(N - n) + n(n - 1)/2} \qbinom{r - 1}{N - n}.
  \end{split}
\end{equation}
Plugging this into \eqref{eq:identical_contour_almost_summation}, we have
\begin{multline}
  P_Y(x_{N - n + 1}(t) > M) = (-1)^{N - n + 1} q^{(N - n + 1)(N - n)/2} \sum^N_{r = N - n + 1} (-1)^r \qbinom{r - 1}{N - n} \\
  \times \sum_{\substack{J' = \{ j'_1, \dotsc, j'_r \} \\ J \subseteq \{ 1, \dotsc, N \} \text{ and } \lvert J' \rvert = r}} q^{(j'_1 + \dotsb + j'_r) - (N - n + 1)r} \\
  \times \dashint_C \frac{dw_{j'_1}}{w_{j'_1}} \dotsi \dashint_C \frac{dw_{j'_r}}{w_{j'_r}} B_r(w_{j'_1}, \dotsc, w_{j'_r}) \prod^r_{i = 1} \left[ \prodprime^M_{k = y_{j_i}} \left( \frac{b_k}{b_k - w_{j'_i}} \right) e^{-w_{j'_i} t} \right],
\end{multline}
and prove \eqref{eq:distr_X_m}.

\subsection{Infinite system with step initial condition}

In this subsection we prove Theorem \ref{theorem-step}. First consider the $q$-TAZRP model with $N$ particles starting at $Y = 0^N = (0, 0, \dotsc, 0)$. By \eqref{eq:distr_X_m}, for a nonnegative integer $M$,
\begin{equation} \label{eq:finite_n^N}
  \mathbb{P}_{0^N}(x_m(t) > M) = (-1)^m q^{m(m - 1)/2} \sum^N_{r = m} S_N(M, t; r, m),
\end{equation}
where
\begin{equation}
  \begin{split}
  S_N(M, t; r, m) = {}& (-1)^r s_N(r, m) \dashint_{\Gamma} \frac{d z_1}{z_1} \dotsi \dashint_{\Gamma} \frac{d z_r}{z_r} B_r(z_1, \dotsc, z_r) \prod^r_{j = 1} \left[ \prod^M_{k = 0} \left( \frac{b_k}{b_k - z_j} \right) e^{-z_j t} \right] \\
  = {}& (-1)^r s_N(r, m) \Prob_{0^r}(x_r(t) > M),
  \end{split}
\end{equation}
such that $\Gamma$ is the contour in \eqref{eq:op_K_Mt}, and
\begin{equation}
  s_N(r, m) = \left( \sum_{J' = \{ j'_1, \dotsc, j'_r \} \subseteq \{ 1, \dotsc, N \} \text{ and } \lvert J' \rvert = r} q^{(j'_1 + \dotsb + j'_r) - mr} \right) \qbinom{r - 1}{m - 1}.
\end{equation}

To prove Theorem \ref{theorem-step}, we take the $N \to \infty$ limit of \eqref{eq:finite_n^N}. Since $\Prob_{0^r}(x_r(t) > M)$ is a number between $0$ and $1$ that is independent of $N$, we only need to take the limit of $s_N(r, m)$ as $N \to \infty$. Since $s_N(r, m) > 0$ for all $N, r, m$, the convergence of the series as $N \to \infty$ is absolute, so we can freely change the order of summation.

We consider first the $m = 1$ case of Theorem \ref{theorem-step}, and have in that case
\begin{equation}
  s_N(r, 1) = \sum_{J' = \{ j'_1, \dotsc, j'_r \} \subseteq \{ 1, \dotsc, N \} \text{ and } \lvert J' \rvert = r} q^{(j'_1 + \dotsb + j'_r) - r}.
\end{equation}
As $N \to \infty$, we have that for all $r$
\begin{equation} \label{eq:limit_s_n(r,1)}
  \lim_{N \to \infty} s_N(r, 1) = \sum_{1\leq j'_1 < j'_2 < \dotsb < j'_r < \infty} q^{(j'_1 + \dotsb + j'_r) - r} = \frac{q^{r(r + 1)/2 - r}}{(1 - q)(1 - q^2) \dotsm (1 - q^r)} = \frac{q^{r(r - 1)/2}}{(1 - q)^r [r]_q!}.
\end{equation}
Thus we have that for any fixed $M$ and $r$, as $N \to \infty$, we have
\begin{equation} \label{eq:limiting_distr_x_1(t)}
  \Prob_{0^{\infty}}(x_1(t) > M) = \lim_{N \to \infty} \Prob_{0^N}(x_1(t) > M) = -\sum^{\infty}_{r = 1} (-1)^r \frac{q^{r(r - 1)/2}}{(1 - q)^r [r]_q!} \Prob_{0^r}(x_r(t) > M).
\end{equation}
Next, we use \cite[Proposition 3.9]{Borodin-Corwin13} and express
\begin{multline}
  (-1)^r q^{r(r - 1)/2} \Prob_{0^r}(x_r(t) > M) \\
  = \frac{[r]_q!}{r!} (1 - q^{-1})^r \dashint_{\Gamma} dz_1 \dotsi \dashint_{\Gamma} dz_r \det \left( \frac{1}{z_j/q - z_k} \right)^r_{j, k = 1} \prod^r_{j = 1} \left[ \prod^M_{k = 0} \left( \frac{b_k}{b_k - z_j} \right) e^{-z_j t} \right],
\end{multline}
and then use \cite[Proposition 3.16]{Borodin-Corwin13} to further express
\begin{equation}
  \begin{split}
    & \sum^{\infty}_{r = 0} (-1)^r \frac{q^{r(r - 1)/2}}{(1 - q)^r [r]_q!} \Prob_{0^r}(x_r(t) > M) \\
    = {}& \sum^{\infty}_{r = 0} \frac{1}{r!} \dashint_{\Gamma} dz_1 \dotsi \dashint_{\Gamma} dz_r \det \left( \frac{-q^{-1}}{z_j/q - z_k} \prod^M_{k = 0} \left( \frac{b_k}{b_k - z_j} \right) e^{-z_j t} \right)^r_{j, k = 1} \\
    = {}& \det(I + K_{M, t}),
  \end{split}
\end{equation}
where $K_{M, t}$ is the integral operator, from $L^2(C)$ to $L^2(C)$, defined in \eqref{eq:op_K_Mt}. Hence by \eqref{eq:limiting_distr_x_1(t)} we conclude that
\begin{equation}
  \Prob_{0^\infty}(x_1(t) \leq M) = 1 - \lim_{N \to \infty} \Prob_{0^N}(x_1(t) > M) = \det(I + K_{M, t}),
\end{equation}
which is the $m = 1$ case of \eqref{theorem-1}.

By the same method, we can prove Theorem \ref{theorem-step} for general $m$. Similar to \eqref{eq:limit_s_n(r,1)}, we have
\begin{equation}
  \begin{split}
    \lim_{N \to \infty} s_N(r, m) = {}& \left( \sum_{J' = \{ j'_1, \dotsc, j'_r \} \subseteq \{ 1, \dotsc, N \} \text{ and } \lvert J' \rvert = r} q^{(j'_1 + \dotsb + j'_r) - mr} \right) \qbinom{r - 1}{m - 1} \\
    = {}& \frac{q^{r(r - 1)/2}}{(1 - q)^r [r]_q!} q^{-(m - 1)r} \qbinom{r - 1}{m - 1}.
  \end{split}
\end{equation}
By the $q$-binomial theorem (See \cite[Corollary 10.2.2(d)]{Andrews-Askey-Roy99})
\begin{equation}
  \sum^{\infty}_{r = m} q^{-(m - 1)r} \qbinom{r - 1}{m - 1} \xi^r =  q^{-m(m - 1)} \xi^m \frac{1}{(1 - \xi)(1 - q^{-1}\xi) \dotsm (1 - q^{-(m - 1)} \xi)},
\end{equation}
where the convergence is for $\lvert \xi \rvert < q^{m - 1}$, so we have that
\begin{equation} \label{eq:step_init_contour_proof}
  \begin{split}
    & \Prob_{0^{\infty}}(x_m(t) > M) \\
    = {}& \lim_{N \to \infty} \Prob_{0^N}(x_m(t) > M) \\
    = {}& (-1)^m q^{m(m - 1)/2} \sum^{\infty}_{r = m} (-1)^r \frac{q^{r(r - 1)/2}}{(1 - q)^r [r]_q!} q^{-(m - 1)r} \qbinom{r - 1}{m - 1} \Prob_{0^r}(x_r(t) > M) \\
    = {}& \frac{(-1)^m q^{-m(m - 1)/2}}{2\pi i} \oint_0 \frac{d\xi}{\xi} \left[ \sum^{\infty}_{r = 0} \frac{\xi^{-r}}{r!} \dashint_{\Gamma} dz_1 \dotsi \dashint_{\Gamma} dz_r \det \left( \frac{-q^{-1}}{z_j/q - z_k} \prod^M_{k = 1} \left( \frac{b_k}{b_k - z_j} \right) e^{-z_j t} \right)^r_{j, k = 1} \right] \\
    & \phantom{\frac{(-1)^m q^{-m(m - 1)/2}}{2\pi i} \oint_0} \times \frac{\xi^m}{(1 - \xi)(1 - q^{-1}\xi) \dotsm (1 - q^{-(m - 1)}\xi)} \\
    = {}& \frac{(-1)^m q^{-m(m - 1)/2}}{2\pi i} \oint_0 \frac{d\xi}{\xi} \det(I + \xi^{-1} K_{M, t}) \frac{\xi^m}{(1 - \xi)(1 - q^{-1}\xi) \dotsm (1 - q^{-(m - 1)}\xi)},
  \end{split}
\end{equation}
where the $\xi$ contour lies in the region $\lvert \xi \rvert < q^{m - 1}$. Finally we take the change of variables $\zeta = \xi^{-1}$, and have \eqref{theorem-1} for general $m$, where the $\zeta$ contour lies out of the region $\lvert \zeta \rvert > q^{1 - m}$. We complete the proof by noting that the $\xi$ contour in \eqref{eq:step_init_contour_proof} can be taken that the corresponding $\zeta$ contour is the contour $C$ given in Theorem \ref{theorem-step}.

\section{Asymptotics} \label{sec:asymptotic}

The main part of this section is the proof of Theorem \ref{thm:asymp}, and in the end of this section we prove Corollary \ref{cor:TW_formula}.

To prove Theorem \ref{thm:asymp}, we assume that $M$, $t$ and $b_k$ are specified as in \eqref{eq:parameter_specification}, depending on constants $\tau$, $l$, $\beta_0, \dotsc, \beta_l$, and $n \to \infty$. Then we follow the approach in \cite{Tracy-Widom08} and \cite{Tracy-Widom09}, and prove the following two limiting properties of the integral operator $K_{M, t}$ in Theorem \ref{theorem-step}:
\begin{enumerate}[label=(\roman*)]
\item \label{enu:step_1}
  \begin{equation} \label{eq:step_1_conv}
    \tr \left( \left( \left. K_{M, t} \right\rvert_{M = n + l + 1, t = n - \tau \sqrt{n}} \right)^k \right) \to \tr \K^k
  \end{equation}
  for all $k = 1, 2, \dotsc$, where $\K = \K_{\tau; \beta_0, \dotsc, \beta_l}$ whose kernel is defined in \eqref{eq:defn_K(z, w)}.
\item \label{enu:step_2}
  $\left. \det(I + \zeta K_{M, t}) \right\rvert_{M= n + l + 1, \text{ and } t = n - \tau \sqrt{n}}$ is uniformly bounded for large $n$ on compact $\zeta$ sets.
\end{enumerate}
If we have the two properties above, then it is straightforward to check that the convergence in \eqref{eq:uniform_conv_K_to_K} holds uniformly in $\lambda$, and then the contour integral on the right-hand side of \eqref{theorem-1} converges to the contour integral on the right-hand side of \eqref{eq:convergence_of_contour_integral_in_det_K}, if the contour $C$ is specified as in Theorem \ref{theorem-step}. Hence Theorem \ref{thm:asymp} is proved. Below we prove the two properties.

\begin{proof}[Proof of property \ref{enu:step_1}]
  We note that $\K$ is a trace class operator, so its Fredholm determinant is well defined. We denote
  \begin{equation}
    f(w) = w + \log(1 - w) \quad \text{and} \quad g(w) = \prod^l_{k = 0} \frac{\beta_k}{\beta_k - w n^{1/2}}.
  \end{equation}
  Note that $\log(1 - w)$ is not well defined on the whole $\Gamma$, and we take the branch cut of $\log(1 - w)$ at $[1, +\infty)$ and assume its imaginary part is between $-\pi$ and $\pi$. Nevertheless, the exponential of $f(w)$ is holomorphic on $\Gamma$, and
  \begin{equation} \label{eq:K_Mt_in_varphi}
    K_{M, t}(w, w') = \frac{\exp \left[ -n \left( f(w) - \frac{\tau}{\sqrt{n}} w \right) \right]}{qw' - w} g(w).
  \end{equation}
  It is not hard to see that $w = 0$ is the unique critical point for $f(w)$. Inspired by this result and owing to the flexibility of the shape of the contour $\Gamma$, we deform $\Gamma$ into the rectangle with vertices $-n^{-1/2} + i$, $2 + i$, $2 - i$ and $-n^{-1/2} -i$, and the four edges denoted by $\Gamma_1, \Gamma_2, \Gamma_3, \Gamma_4$, as shown in Figure \ref{fig:rect_Gamma}. Here we note that the orientation of $\Gamma_1$ is downward.
  \begin{figure}[htb]
    \centering
    \includegraphics{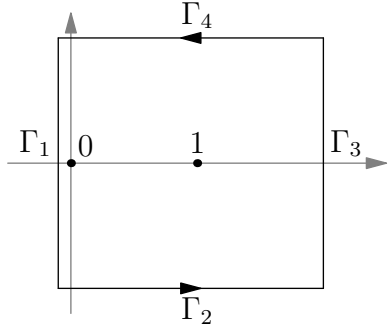}
    \caption{The rectangular shape of $\Gamma$.}
    \label{fig:rect_Gamma}
  \end{figure}
  It is easy to check that $\Re f(w)$ attains its minimum at $w = -n^{-1/2}$ over $\Gamma$. To see it, we only need to consider $\eta \in \Gamma \cap \compC_+$ since $f(\bar{w}) = \overline{f(w)}$. Using that
  \begin{enumerate*}[label=(\roman*)]
  \item
    as $w = -n^{-1/2} + iy \in \Gamma_1 \cap \compC_+$, $\Re f(w) = -n^{-1/2} + \log \lvert 1 - w \rvert$ is a decreasing function of $y$,
  \item
    as $w = x + i \in \Gamma_4$, $\Re f(w)$ is an increasing function of $x$ for $\frac{d}{dx} \Re f(w) = 1 + \Re((1 - w)^{-1}) > 0$, and
  \item
    as $w = 2 + iy \in \Gamma_3 \cap \compC_+$, $\Re f(w)$ is an decreasing function of $y$ with minimum at $\Re f(2) > \Re f(-n^{-1/2})$,
  \end{enumerate*}
  we verify that $-n^{-1/2}$ is the minimum of $\Re f(w)$ on $\Gamma$.

  Next, we denote
  \begin{equation}
    \Gamma_{\loc} = \{ w \in \Gamma_1 \mid \lvert \Im w \rvert < n^{-2/5} \}, \quad \Gamma_{\res} = \Gamma \setminus \Gamma_{\loc}.
  \end{equation}
  By the discussion on $\Re f(w)$ above, it is straightforward to see that
  \begin{equation} \label{eq:Re_f_est_residual}
    \Re \left( f(w) - \frac{\tau}{\sqrt{n}} w \right) \geq \Re f(-n^{-1/2} + i n^{-2/5}) > -n^{-1/2} + c n^{-4/5}, \quad \text{for all $w \in \Gamma_{\res}$},
  \end{equation}
  where $c$ is a small enough positive number, say, $c = 1/10$. Thus the kernel $K_{M, t}(w, w')$, although defined on the whole $\Gamma$, concentrates locally at the middle of $\Gamma_1$, around $0$.

  Let $w = n^{-1/2} u$ and $w' = n^{-1/2} v$. Then for $w, w' = \bigO(n^{-2/5})$, $u, v = \bigO(n^{1/10})$. We have that for all $w \in \compC$
  \begin{equation}
    g(w) = \gamma(u),
  \end{equation}
  where $\gamma$ is defined in \eqref{eq:defn_K(z, w)}. On the other hand, since
  \begin{equation}
    f(0) = f'(0) = 0, \quad \text{and} \quad f''(0) = -1,
  \end{equation}
  we have that for $w \in \Gamma_{\loc}$,
  \begin{equation} \label{eq:nf_asy}
    -n \left( f(w) - \frac{\tau}{\sqrt{n}} w \right) = \frac{u^2}{2} + \tau u^2 + n \bigO(w^3) =  \frac{u^2}{2} + \tau u^2 + \bigO(n^{-1/5}).
  \end{equation}

  Now we are ready to prove \eqref{eq:step_1_conv}. For notational simplicity, we give the detail for the $k = 2$ case, and then explain how to extend the argument to general $k$.

  We have by \eqref{eq:K_Mt_in_varphi} and \cite[Theorem 3.9]{Simon05}
  \begin{equation} \label{eq:two_fold_int_tr_K^2}
    \tr \left( (K_{M, t})^2 \right) = \int_{\Gamma} dw_1 \int_{\Gamma} dw_2 K_{M, t}(w_1, w_2) K_{M, t}(w_2, w_1) = I_{\loc, \loc} + I_{\loc, \res} + I_{\res, \loc} + I_{\res, \res},
  \end{equation}
  where
  \begin{multline}
    I_{*, \star} = \int_{\Gamma_*} dw_1 \int_{\Gamma_{\star}} dw_2 \frac{\exp \left[ -n \left( f(w_1) + f(w_2) - \frac{\tau}{\sqrt{n}} (w_1 + w_2) \right) \right]}{(w_1 - qw_2)(w_2 - qw_1)} g(w_1) g(w_2), \\
    *, \star = \loc \text{ or } \res.
  \end{multline}
  By \eqref{eq:nf_asy}, we have
  \begin{equation} \label{eq:approx_I_loc_loc}
    \begin{split}
      I_{\loc, \loc} = {}& \int^{-1 - n^{1/10} i}_{-1 + n^{1/10} i} du \int^{-1 - n^{1/10} i}_{-1 + n^{1/10} i} dv \frac{e^{(u^2 + v^2)/2 + \tau(u + v)}}{(u - qv)(v - qu)} \gamma(u) \gamma(v) (1 + \bigO(n^{-1/5})) \\
      = {}& \int_{\Gamma_{\infty}} du \int_{\Gamma_{\infty}} dv \frac{e^{(u^2 + v^2)/2 + \tau(u + v)}}{(u - qv)(v - qu)} \gamma(u) \gamma(v) + \bigO(n^{-1/5}),
    \end{split}
  \end{equation}
  where $\Gamma_{\infty} = \{ -1 - iy \mid y \in \realR \}$ is defined in Theorem \ref{thm:asymp}. By \eqref{eq:Re_f_est_residual}, we have the estimate
  \begin{equation} \label{eq:est_I_res_res}
    I_{\res, \res} = o(e^{-2c n^{1/5}}),
  \end{equation}
  and using \eqref{eq:Re_f_est_residual} and \eqref{eq:nf_asy} together, we derive that
  \begin{equation} \label{eq:est_I_res_loc}
    I_{\loc, \res} = o(e^{-c n^{1/5}}), \quad I_{\res, \loc} = o(e^{-c n^{-1/5}}).
  \end{equation}
  On the other hand, by the definition of $\K$ in \eqref{eq:defn_K(z, w)},
  \begin{equation} \label{eq:tr_K^2}
    \tr \K^2 = \int_{\Gamma_{\infty}} dz \int_{\Gamma_{\infty}} dw \frac{e^{(z^2 + w^2)/2 + \tau(z + w)}}{(z - qw)(w - qz)} \gamma(z) \gamma(w).
  \end{equation}
  Comparing \eqref{eq:tr_K^2} with \eqref{eq:approx_I_loc_loc}, \eqref{eq:est_I_res_res} and \eqref{eq:est_I_res_loc}, we prove the $k = 2$ case of \eqref{eq:step_1_conv}. For general $k$, we need to replace the two-fold integral in \eqref{eq:two_fold_int_tr_K^2} by a $k$-fold integral as in \cite[Theorem 3.9]{Simon05}, and then similarly decompose the $k$-fold integral into the sum of $2^k$ terms, such that each term is a $k$-fold iterated integral with each integral domain being $\Gamma_{\loc}$ or $\Gamma_{\res}$. Then the argument for the $k = 2$ case works for the general $k$ case.
\end{proof}

\begin{proof}[Proof of property \ref{enu:step_2}]
  We note that the integral operator $K_{M, t}$ is a trace class operator as well as a Hilbert--Schmidt operator, so we can write
  \begin{equation}
    \det(I - \lambda K_{M, t}) = \dettwo(1 - \lambda K_{M, t}) e^{\tr K_{M, t}},
  \end{equation}
  where $\dettwo$ is defined in \cite[Chapter 9]{Simon05}. First, by \cite[Theorem 3.9]{Simon05}, we have
  \begin{equation}
    \tr K_{M, t} = \frac{1}{2\pi i} \int_{\Gamma} K_{M, t}(w, w) dw.
  \end{equation}
  Thus as $M = n + l + 1$, $t = n - \tau\sqrt{n}$ and $n \to \infty$, we have, like the computation of $\tr((K_{M, t})^2)$ in the proof of property \ref{enu:step_1},
  \begin{equation}
    \begin{split}
      \tr K_{M, t} = {}& \frac{1}{2\pi i} \int_{\Gamma} \frac{\exp \left[ -n \left( f(w) - \frac{\tau}{\sqrt{n}} w \right) \right]}{(1 - q)w} g(w) dw \\
      = {}& \frac{1}{2\pi i} \int_{\Gamma_{\loc}} \frac{\exp \left[ -n \left( f(w) - \frac{\tau}{\sqrt{n}} w \right) \right]}{(1 - q)w} g(w) dw + o(e^{-c n^{1/5}}) \\
      = {}& \frac{1}{2\pi i} \int^{-1 - n^{1/10} i}_{-1 + n^{1/10}i} \frac{e^{u^2/2 + \tau u}}{(1 - q)u} \gamma(u) (1 + \bigO(n^{-1/5})) du + o(e^{-c n^{1/5}}) \\
      =  {}& \frac{1}{2\pi i} \int_{\Gamma_{\infty}} \frac{e^{u^2/2 + \tau u}}{(1 - q)u} \gamma(u) du + \bigO(n^{-1/5}) \\
      = {}& C + \bigO(n^{-1/5}),
    \end{split}
  \end{equation}
  where $c$ is the same as in \eqref{eq:Re_f_est_residual} and the $C$ is equal to the contour integral on $\Gamma_{\infty}$, and we do not compute it explicitly.

  On the other hand, by \cite[Theorem 9.2(b)]{Simon05}, we have
  \begin{equation}
    \lvert \dettwo(1 - \lambda K_{M, t}) \rvert \leq \exp( \lvert \lambda \rvert^2 \lVert K_{M, t} \rVert^2_2),
  \end{equation}
  where $\lVert \cdot \rVert_2$ is the Hilbert--Schmidt norm. (The constant $\Gamma_n$ in \cite[Theorem 9.2(b)]{Simon05} can be taken as $1$.) Thus we only need to show that $\lVert K_{M, t} \rVert_2$ is bounded as $M = n + l + 1$, $t = n - \tau\sqrt{n}$ and $n \to \infty$. To this end, we use \cite[Theorem 2.11]{Simon05}, and have
  \begin{equation}
    \lVert K_{M, t} \rVert^2_2 = \frac{1}{4\pi^2} \int_{\Gamma} \lvert dw' \rvert \int_{\Gamma} \lvert dw \rvert \left\lvert \frac{\varphi(w)}{qw' - w} g(w') \right\rvert^2.
  \end{equation}
  As $n \to \infty$, we have, by the concentration of $K_{M, t}(w, w')$ over $\Gamma$ around $0$,
  \begin{equation} \label{eq:est_K_Mt_HS_norm}
    \begin{split}
      \left\lVert K_{M, t} \right\rVert^2_2 = {}& \frac{1}{4\pi^2} \int_{\Gamma} \lvert dw' \rvert \int_{\Gamma} \lvert dw \rvert \left\lvert \frac{\exp \left[ -n \left( f(w) - \frac{\tau}{\sqrt{n}} w \right) \right]}{qw' - w} g(w) \right\rvert^2 \\
       = {}& \frac{1}{4\pi^2} \int_{\Gamma_{\loc}} \lvert dw' \rvert \int_{\Gamma_{\loc}} \lvert dw \rvert \left\lvert \frac{\exp \left[ -n \left( f(w) - \frac{\tau}{\sqrt{n}} w \right) \right]}{qw' - w} g(w) \right\rvert^2 + o(e^{-cn^{1/5}}) \\
       = {}& \frac{1}{4\pi^2} \int^{-1 - n^{1/10} i}_{-1 + n^{1/10} i} \lvert dv \rvert \int^{-1 - n^{1/10} i}_{-1 + n^{1/10} i} \lvert du \rvert \left\lvert \frac{e^{u^2/2 + \tau u}}{qv - u} \gamma(u) \right\rvert^2 (1 + \bigO(n^{-1/5})) + o(e^{-cn^{1/5}}) \\
       = {}& \frac{1}{4\pi^2} \int_{\Gamma_{\infty}} \lvert dv \rvert \int_{\Gamma_{\infty}} \lvert du \rvert \left\lvert \frac{e^{u^2/2 + \tau u}}{qv - u} \gamma(u) \right\rvert^2 + \bigO(n^{-1/5}).
    \end{split}
  \end{equation}
  We are not going to evaluate the integral in \eqref{eq:est_K_Mt_HS_norm} explicitly, but it is clear that $\lVert K_{M, t} \rVert^2_2$ is bounded as $M, t, b_k$ are specified in \eqref{eq:parameter_specification} and $n \to \infty$. Hence we finish the proof.
\end{proof}

\begin{remark} \label{rmk:technical}
  Our proof to Theorem \ref{theorem-step} is similar to the proof of \cite[Theorem 2]{Tracy-Widom09} that obtains analogous limiting distributions for the leftmost eigenvalues in the ASEP with step initial condition. The argument in our proof is simpler than that in \cite{Tracy-Widom09}, because our integral operator $K_{M, t}$ is on $L^2(\Gamma)$ such that $\Gamma$ encloses the pole $0$, while the counterpart integral operator $K_2$ defined in \cite[Section 2]{Tracy-Widom09} is defined on $L^2(\gamma)$ such that the contour $\gamma$ does not enclose the pole $0$. Analogously, one can compute the limiting distribution of $x_m(t)$ as both $m, t \to \infty$ by the method in the proof of \cite[Theorem 3]{Tracy-Widom09}, and some steps can be simplified since our contour $\Gamma$ is more friendly to asymptotic analysis than the $\gamma$ in \cite{Tracy-Widom09}.
\end{remark}

\begin{proof}[Proof of Corollary \ref{cor:TW_formula}]
  We only need to prove part \ref{enu:cor_TW_formula_a} by checking \eqref{eq:Mehler_kernel}, and then part \ref{enu:cor_TW_formula_b} follows as we plug the Fredholm determinants $\det(I + \zeta \K_{\tau; -})$ and $\det(I + \zeta \hat{K} \chi_{(\tau(1 + q)/(1 - q), \infty)})$ into the countour integral formulas for the $m$-th right-most particle in $q$-TAZRP (the right-hand side of our \eqref{eq:convergence_of_contour_integral_in_det_K}) and for the $m$-th left-most particle in ASEP (the right-hand side of \cite[Formula (2)]{Tracy-Widom09}) respectively.

  To check \eqref{eq:Mehler_kernel}, we note that for all $z, w \in \Gamma_{\infty}$, since $\Re(w - qz) = q - 1 < 0$,
  \begin{equation}
    \frac{1}{w - qz} = -\int^{\infty}_0 e^{(w - qz)\xi} d\xi,
  \end{equation}
  and then
  \begin{equation}
    \K_{\tau; -}(z, w) = -\int^{\infty}_0 e^{-qz \xi} e^{\frac{w^2}{2} + w(\xi + \tau)} d\xi.
  \end{equation}
  Hence we define the integral operators $A: L^2(\Gamma_{\infty}) \to L^2(\realR_+)$ and $B: L^2(\realR_+) \to L^2(\Gamma_{\infty})$, with kernels
  \begin{equation}
    A(\xi, w) = e^{\frac{w^2}{2} + w(\xi + \tau)}, \quad B(z, \xi) = e^{-qz \xi}.
  \end{equation}
  We have that $\K_{\tau; -} = -AB$. If we define $\tilde{\K}_{\tau; -} = -BA$, an operator from $L^2(\realR_+)$ to $L^2(\realR_+)$, then $\det(I + \K_{\tau; -}) = \det(I + \tilde{\K}_{\tau; -})$. Here $\tilde{\K}_{\tau; -}$ is an integral operator with kernel
  \begin{equation}
    \tilde{\K}_{\tau; -}(\xi, \eta) = -\frac{1}{2\pi i} \int_{\Gamma_{\infty}} e^{\frac{z^2}{2} + z(\xi + \tau)} e^{-qz \eta} dz = \frac{1}{\sqrt{2\pi}} e^{-\frac{1}{2} (\xi - q\eta + \tau)^2}.
  \end{equation}
  By change of variables, we hav3 that
  \begin{equation}
    \begin{split}
      \det(I + K_{\tau; -}) = {}& \det(I + \tilde{K}_{\tau; -}) \\
      = {}& \det \left( I - \chi_{(0, \infty)}(\xi) \frac{1}{\sqrt{2\pi}} e^{-\frac{1}{2} (\xi - q\eta + \tau)^2} \chi_{(0, \infty)}(\eta) \right) \\
      = {}& \det \left( I - \chi_{(\tau/(1 - q), \infty)}(\xi) \frac{1}{\sqrt{2\pi}} e^{-\frac{1}{2} (\xi - q\eta)^2} \chi_{(\tau/(1 - q), \infty)}(\eta) \right) \\
      = {}& \det \left( I - \chi_{(\tau(1 + q)/(1 - q), \infty)}(\xi) \frac{1}{\sqrt{2\pi}(1 + q)} e^{-\frac{(\xi - q\eta)^2}{2(1 + q)^2}} \chi_{(\tau(1 + q)/(1 - q), \infty)}(\eta) \right).
    \end{split}
  \end{equation}
  Noting that
  \begin{equation}
    \frac{1}{\sqrt{2\pi}(1 + q)} e^{-\frac{(\xi - q\eta)^2}{2}} = \frac{e^{(1 - q^2)\eta^2/4}}{e^{(1 - q^2)\xi^2/4}} \hat{K}(\xi, \eta),
  \end{equation}
  where $\hat{K}$ is defined in \eqref{eq:defn_hat_K}, and that Fredholm determinant is invariant under conjugation, we prove \eqref{eq:Mehler_kernel}.
\end{proof}
\appendix

\section{Limiting distributions of $x_m(t)$ as $m, t \to \infty$ and $t/m \to c$} \label{sec:TW_limit}

We first fix some notations following \cite{Barraquand15}.
Let
\begin{equation}
  \Psi_q(z) = \frac{\partial}{\partial z} \log \Gamma_q(z), \quad \text{where} \quad \Gamma_q(z) = (1 - q)^{1 - z} \frac{(q; q)_{\infty}}{(q^z; q)_{\infty}}.
\end{equation}
Then for $\theta > 0$, let
\begin{align}
  \kappa  = \kappa(q, \theta) = {}& \frac{\Psi'_q(\theta)}{(\log q)^2 q^{\theta}}, \\
  f = f(q, \theta) = {}& \frac{\Psi'_q(\theta)}{(\log q)^2} - \frac{\Psi_q(\theta)}{\log q} - \frac{\log(1 - q)}{\log q}, \\
  \chi = \chi(q, \theta) = {}& \frac{1}{2}(\Psi'_q(\theta) \log q - \Psi''_q(\theta)).
\end{align}
Furthermore, fix $\alpha \in (0, 1]$, we define
\begin{align}
  g = g(q, \theta) = {}& \frac{\Psi'_q(\theta)}{(\log q)^2} \frac{\alpha}{q^{\theta}} - \frac{\Psi_q(\log_q \alpha)}{\log q} - \frac{\log(1 - q)}{\log q}, \\
  \sigma = \sigma(q, \theta) = {}& \Psi'_q(\theta) \frac{\alpha}{q^{\theta}} - \Psi'_q(\log_q \alpha).
\end{align}
We note that as $\theta$ runs over $(0, +\infty)$, $\chi(q, \theta) > 0$, and $\kappa(q, \theta)$ decreases monotonically from $+\infty$ to $(1 - q)^{-1}$. Recall that in \eqref{eq:relation_q_TASEP_q_TAZRP_prob_distr} we introduced the notation $\Prob^{\qTASEP}_{\step}(y_n(t) \geq x)$ to represent the distribution of the $n$-th particle in the $q$-TASEP with step initial condition.
\begin{proposition} \cite[Theorem 1]{Barraquand15} \label{prop:Barraquand}
  Suppose $b_k = 1$ for all $k > m$, $\alpha = \min(b_1, \dotsc, b_m) \in (0, 1]$, and $k$ out of $b_1, \dotsc, b_m$ are equal to $\alpha$. Let $c \in \realR$ be a constant.
  \begin{enumerate}
  \item
    For $\theta \in (\log_q \alpha, +\infty)$, we have
    \begin{multline}
      \lim_{n \to \infty} \Prob^{\qTASEP}_{\step} \left( y_n \left( (1 - q)\kappa n + (1 - q) q^{-\theta} c n^{2/3} \right) \vphantom{(f - 1)n + cn^{2/3} - c^2 \frac{(\log q)^3}{4\chi} n^{1/3} + \frac{\chi^{1/3}}{\log q} \xi n^{1/3}} \right. \\
      \geq \left. (f - 1)n + cn^{2/3} - c^2 \frac{(\log q)^3}{4\chi} n^{1/3} + \frac{\chi^{1/3}}{\log q} \xi n^{1/3} \right) = F_{\GUE}(\xi),
    \end{multline}
  \item
    For $\theta = \log_q \alpha$, we have
    \begin{multline}
      \lim_{n \to \infty} \Prob^{\qTASEP}_{\step} \left( y_n \left( (1 - q)\kappa n + (1 - q) q^{-\theta} c n^{2/3} \right) \vphantom{(f - 1)n + cn^{2/3} - c^2 \frac{(\log q)^3}{4\chi} n^{1/3} + \frac{\chi^{1/3}}{\log q} \xi n^{1/3}} \right. \\
      \geq \left. (f - 1)n + cn^{2/3} - c^2 \frac{(\log q)^3}{4\chi} n^{1/3} + \frac{\chi^{1/3}}{\log q} \xi n^{1/3} \right) = F_{\BBP, k, \mathbf{b}}(\xi),
    \end{multline}
    where $\mathbf{b} = (b, \dotsc, b)$ with $b = c(\log q)^2/(2\chi^{2/3})$.
  \item
    If $\theta \in (0, \log_q \alpha)$, we have
    \begin{multline}
      \lim_{n \to \infty} \Prob^{\qTASEP}_{\step} \left( y_n \left( (1 - q)\kappa n + (1 - q) \alpha^{-1} c n^{1/2} \right) \geq (g - 1)n + c n^{1/2} + \frac{\sigma^{1/2}}{\log q} \xi n^{1/3} \right)  \\
      = G_k(\xi).
    \end{multline}
  \end{enumerate}
  Here $F_{\GUE}$ is the GUE Tracy--Widom distribution, $F_{\BBP, k, \mathbf{b}}$ is the BBP (Baik--Ben Arous--\Peche) distribution with rank $k$, and $G_k$ is the distribution of the largest eigenvalues of a $k \times k$ GUE. Their precise formulas are given in \cite[Definition 3]{Barraquand15}. The phase transition of the three distributions in \cite{Barraquand15} was studied first in the spiked Wishart ensemble in random matrix theory \cite{Baik-Ben_Arous-Peche05}.
\end{proposition}

Hence by relation \eqref{eq:relation_q_TASEP_q_TAZRP_prob_distr}, we have the corresponding results for the $q$-TAZRP with step initial condition.
\begin{corollary}
  Suppose $b_k = 1$ for all $k > m$, $\alpha = \min(b_1, \dotsc, b_m) \in (0, 1]$, and $k$ out of $b_1, \dotsc, b_m$ are equal to $\alpha$. Let $c \in \realR$ be a constant.
  \begin{enumerate}
  \item
    For $\theta \in (\log_q \alpha, +\infty)$, we have
    \begin{multline}
      \lim_{n \to \infty} \Prob_{0^{\infty}} \left( x_{\left\lceil (f - 1)n + cn^{2/3} - c^2 \frac{(\log q)^3}{4\chi} n^{1/3} + \frac{\chi^{1/3}}{\log q} \xi n^{1/3} \right\rceil} \left( (1 - q)\kappa n + (1 - q) q^{-\theta} c n^{2/3} \right) > n \right) \\
      = F_{\GUE}(\xi),
    \end{multline}
  \item
    For $\theta = \log_q \alpha$, we have
  \begin{multline}
    \lim_{n \to \infty} \Prob_{0^{\infty}} \left( y_{\left\lceil fn + cn^{2/3} - c^2 \frac{(\log q)^3}{4\chi} n^{1/3} + \frac{\chi^{1/3}}{\log q} \xi n^{1/3} \right\rceil} \left( (1 - q)\kappa n + (1 - q) q^{-\theta} c n^{2/3} \right) > n \right) \\
    = F_{\BBP, k, \mathbf{b}}(\xi),
  \end{multline}
  where $\mathbf{b} = (b, \dotsc, b)$ with $b = c(\log q)^2/(2\chi^{2/3})$.
\item
  If $\theta \in (0, \log_q \alpha)$, we have
\begin{equation}
  \lim_{n \to \infty} \Prob_{0^{\infty}} \left( y_{\left\lceil gn + c n^{1/2} + \frac{\sigma^{1/2}}{\log q} \xi n^{1/3} \right\rceil} \left( (1 - q)\kappa n + (1 - q) \alpha^{-1} c n^{1/2} \right) > n \right) = G_k(\xi).
\end{equation}
  \end{enumerate}
\end{corollary}


\end{document}